\documentclass[12pt,amstex,leqno]{amsart}
\usepackage{amssymb}

\input epsf.tex

\usepackage[pagewise]{lineno}
\usepackage{epsfig,subfigure,fancybox}
\usepackage{amsmath}
\usepackage{amssymb}
\usepackage{amsfonts}
\usepackage{graphicx}
\usepackage{dsfont}
\usepackage[subfigure]{graphfig}
\usepackage{float}

\input epsf.tex

\usepackage{amsmath}
\usepackage{amssymb}
\usepackage{amsfonts}
\usepackage{graphicx}
\usepackage{dsfont}
\usepackage[subfigure]{graphfig}
\usepackage{lscape}
\usepackage[subfigure]{graphfig}
\usepackage{float}
\usepackage{epsfig,subfigure,fancybox}
\usepackage{amsmath}
\usepackage{tikz}
\usepackage{amssymb}
\usepackage{amsfonts}
\usepackage{graphicx}
\usepackage{dsfont}
\usepackage[subfigure]{graphfig}

\newtheorem{thmm}{Theorem}
\newtheorem{lemm}{Lemma}

\newcounter{neweqn}

\newcommand{\beq}[1]{\begin{equation} \refstepcounter{neweqn} \label{#1}}
\newcommand{\eeq}{\end{equation}}
\newcommand{\bed}{\begin{displaymath}}
\newcommand{\eed}{\end{displaymath}}
\newcommand{\bedd}{\bed\begin{array}{l}}
\newcommand{\eedd}{\end{array}\eed}

\newcommand{\disp}{\displaystyle}
\newcommand{\A}{{\mathcal A}}

\newcommand{\Bb}{\beta_{\rm b}}
\newcommand{\Bs}{\beta_{\rm s}}

\newcommand{\rr}{{\hbox{{\rm I}{\kern -0.22em}{\rm R}}}}

\newcommand{\sd}{\operatorname{d}}
\newcommand{\dt}{\operatorname{dt}}
\def\({\left(}
\def\){\right)}

\def \fA {{\mathcal A}}

\def \L {{\mathcal L}}

\def \R {{\mathbb R}}

\begin{document}

\title{Optimal Strategies for Round-Trip Pairs Trading Under Geometric Brownian Motions}
\author{E. Crawford Das}
\address{({\rm E. Crawford Das}) Department of Mathematics, University of Georgia,
Athens, GA 30602, USA,  Emily.Crawford@uga.edu}
\author{J. Tie}
\address{({\rm J. Tie}) Department of Mathematics, University of Georgia,
Athens, GA 30602, USA,  jtie@uga.edu}
\author{Q. Zhang}
\address{({\rm Q. Zhang}) Department of Mathematics, University of Georgia,
Athens, GA 30602, USA, qz@uga.edu}


\maketitle
\begin{abstract}
This paper is concerned with an optimal strategy for simultaneously trading a pair of stocks. The idea of pairs trading is to monitor their price movements and compare their relative strength over time. A pairs trade is triggered by the divergence of their prices and consists of a pair of positions to short the strong stock and to long the weak one. Such a strategy bets on the reversal of their price strengths. A round-trip trading strategy refers to opening and closing such a pair of security positions. Typical pairs-trading models usually assume a difference of the stock prices satisfies a mean-reversion equation. However, we consider the optimal pairs-trading problem by allowing the stock prices to follow general geometric Brownian motions. The objective is to trade the pairs over time to maximize an overall return with a fixed commission cost for each transaction. Initially, we allow the initial pairs position to be either long or flat.  We then consider the problem when the initial pairs position may be long, flat, or short.  In each case, the optimal policy is characterized by threshold curves obtained by solving the associated HJB equations.




\end{abstract}

\newpage

\baselineskip22pt
\section{Introduction}

\par	This paper is concerned with an optimal strategy for simultaneously trading a pair of stocks. The purpose of pairs trading is to hedge the risk associated with buying and holding shares of one stock by selling shares of a related stock.  The idea of pairs trading is to track the prices of the two stocks that follow roughly the same trajectory over time.  A pairs trade is triggered by the divergence of their prices and consists of a pair of positions to short the strong stock and to long the weak one. Such a strategy bets on the reversal of their price strengths. Pairs trading, which was pioneered by quantitative researchers at brokerage firms in the 1980s, is beneficial, because it can be profitable under any market circumstances \cite{Gatev}.  A round-trip trading strategy refers to opening and closing such a pair of security positions. Typical pairs-trading models usually assume a difference of the stock prices satisfies a mean-reversion equation. However, we consider the optimal pairs-trading problem by allowing the stock prices to follow general geometric Brownian motions as in \cite{JOptim}. One benefit of this model is that it does not specificy any relationship between the pairs of stocks or require them to satisify any measure of correlation, thus allowing for greater possibilities in the choice of pairs.  The Brownian motion, whose sample path is a random walk, encodes the assumption that it is impossible to accurately predict the change in the price of a stock from day to day.  Our objective is to trade the pairs over time to maximize an overall return with a fixed commission cost for each transaction. Initially, we allow the initial pairs position to be either long or flat.  We then consider the problem when the initial pairs position may be long, flat, or short.  In each case, the optimal policy is characterized by threshold curves obtained by solving the associated Hamilton-Jacobi-Bellman (HJB) equations.
\section{Problem Formulation}
Consider two stocks, $\textbf{S}^1$ and $\textbf{S}^2$.  Let $\{X_t^1, t\ge 0\}$ denote the prices of the stock $\textbf{S}^1$, and let $\{X_t^2, t\ge0\}$ denote the prices of the stock $\textbf{S}^2$.  They satisfy the following stochastic differential equation:
\beq{SDE}
\sd
\begin{pmatrix}
 X_t^1 \\
X_t^2
\end{pmatrix}
=
\begin{pmatrix}
 X_t^1 & \\
& X_t^2
\end{pmatrix}
\left[
\begin{pmatrix}
 \mu_1 \\
\mu_2
\end{pmatrix}
\dt +
\begin{pmatrix}
 \sigma_{11} & \sigma_{12} \\
 \sigma_{21} & \sigma_{22} 
\end{pmatrix}
\sd
\begin{pmatrix}
W_t^1 \\
W_t^2
\end{pmatrix}
\right]
\eeq
where $\mu_i$, $i=1, 2$ are the return rates, $\sigma_{ij}$, $i, j = 1,2$ are the volatility constants, and $\left(W_t^1, W_t^2\right)$ is a 2-dimensional standard Brownian motion.

In this paper, we consider a round-trip pairs trading strategy.  Intially, we assume the pairs position, which we will denote $\textbf{Z}$, consists of a one-share long position in stock $\textbf{S}^1$ and a one-share short position in stock $\textbf{S}^2$.  We consider the case that the net position may initially be long (with one share of $\textbf{Z}$) or flat (with no stock holdings of either  $\textbf{S}^1$ or $\textbf{S}^2$).  Let $i=0,1$ denote the initial net positions of long and flat, respectively.  If initially we are long ($i=1$), we will close the pairs position $\textbf{Z}$ at some time $\tau_0$ and conclude our trading activity.  Otherwise, if initially we are flat ($i=0$), we will first obtain one share of $\textbf{Z}$ at some time $\tau_1$, and then close pairs position $\textbf{Z}$ at some time $\tau_2\ge\tau_1,$ thus concluding our trading activity.

Let $K$ denote the fixed percentage of transaction costs associate with buying or selling of stocks.  Then given the initial state $(x_1,x_2),$ the initial net position $i=0,1,$ and the decision sequences $\Lambda_1=(\tau_0)$ and $\Lambda_0=(\tau_1,\tau_2)$, the resulting reward functions are
\begin{align}
J_0(x_1,x_2,\Lambda_0)=&~\mathbb{E}\big[e^{-\rho\tau_2}\left(\beta_s X_{\tau_2}^1-\beta_b X_{\tau_2}^2\right)\mathbb{I}_{\{\tau_2<\infty\}}-e^{-\rho\tau_1}\left(\beta_b X_{\tau_1}^1-\beta_s X_{\tau_1}^2\right)\mathbb{I}_{\{\tau_1<\infty\}}\big]\\
J_1(x_1,x_2,\Lambda_1)=&~\mathbb{E}\left[e^{-\rho\tau_0}\left(\beta_s X_{\tau_0}^1-\beta_b X_{\tau_0}^2\right)\mathbb{I}_{\{\tau_0<\infty\}}\right].
\end{align}

Let $V_0(x_1,x_2)=\underset{\Lambda_0}\sup ~J_0(x_1,x_2,\Lambda_0)$ and $V_1(x_1,x_2)=\underset{\Lambda_1}\sup ~J_1(x_1,x_2,\Lambda_1)$ be the associated value functions.

\section{Properties of the Value Functions}
In this section, we establish basic properties of the value functions.
\begin{lemm}
For all $x_1$, $x_2>0$, we have
\begin{align}
0\le  V_0(x_1,x_2)\le 2x_1+2x_2,\\
\beta_sx_1-\beta_bx_2\le  V_1(x_1,x_2)\le x_1.
\end{align}
\end{lemm}
\begin{proof}
Note that for all $x_1$, $x_2>0,$ $V_1(x_1,x_2)\ge J_1(x_1,x_2,\Lambda_1)=\mathbb{E}\left[e^{-\rho\tau_0}\left(\beta_sX_{\tau_0}^1-\beta_bX_{\tau_0}^2\right)\mathbb{I}_{\{\tau_0<\infty\}}\right]$.  In particular, $$V_1(x_1,x_2)\ge J_1(x_1,x_2,0)=\beta_sx_1-\beta_bx_2.$$ 
For all $\tau_0>0,$
$  J_1(x_1,x_2,\Lambda_1) $
\begin{align*}
& =\mathbb{E}\left[e^{-\rho\tau_0}\left(\beta_sX_{\tau_0}^1-\beta_bX_{\tau_0}^2\right)\mathbb{I}_{\{\tau_0<\infty\}}\right]\\
& \le\mathbb{E}\left[e^{-\rho\tau_0}\left(X_{\tau_0}^1-X_{\tau_0}^2\right)\mathbb{I}_{\{\tau_0<\infty\}}\right]\\
& =x_1+\mathbb{E}\left[\int_0^{\tau_0}\left(-\rho+\mu_1\right)e^{-\rho t}X_t^1 \dt \mathbb{I}_{\{\tau_0<\infty\}}\right]-x_2-\mathbb{E}\left[\int_0^{\tau_0}\left(-\rho+\mu_2\right)e^{-\rho t}X_t^2 \dt \mathbb{I}_{\{\tau_0<\infty\}}\right]\\
& \le x_1-x_2-\mathbb{E}\left[\int_0^{\tau_0}\left(-\rho+\mu_2\right)e^{-\rho t}X_t^2 \dt \mathbb{I}_{\{\tau_0<\infty\}}\right]\\
& \le x_1-x_2+\mathbb{E}\left[\int_0^{\infty}\left(\rho-\mu_2\right)e^{-\rho t}X_t^2 \dt \right]\\
& = x_1.
\end{align*}
Also, for all $x_1$, $x_2 >0$,
\begin{align*}
V_0(x_1,x_2)&\ge J_0(x_1,x_2,\Lambda_0)\\
&=\mathbb{E}\big[e^{-\rho\tau_2}\left(\beta_sX_{\tau_2}^1-\beta_bX_{\tau_2}^2\right)\mathbb{I}_{\{\tau_2<\infty\}}- e^{-\rho\tau_1}\left(\beta_bX_{\tau_1}^1-\beta_sX_{\tau_1}^2\right)\mathbb{I}_{\{\tau_1<\infty\}}.
\end{align*} 
 Clearly, $V_0(x_2,x_2)\ge 0$ by definition and taking $\tau_1=\infty.$  Now, for all $0\le\tau_1\le\tau_2,$ $J_0(x_1,x_2,\Lambda_0)$
\begin{align*}
= & ~ \mathbb{E}\big[e^{-\rho\tau_2}\left(\beta_sX_{\tau_2}^1-\beta_bX_{\tau_2}^2\right)\mathbb{I}_{\{\tau_2<\infty\}}\big]-\mathbb{E}\big[e^{-\rho\tau_1}\left(\beta_bX_{\tau_1}^1-\beta_sX_{\tau_1}^2\right)\mathbb{I}_{\{\tau_1<\infty\}}\big]\\
\le &~  \mathbb{E}\big[e^{-\rho\tau_2}X_{\tau_2}^1\mathbb{I}_{\{\tau_2<\infty\}}\big]-\mathbb{E}\big[e^{-\rho\tau_2}X_{\tau_2}^2\mathbb{I}_{\{\tau_2<\infty\}}\big]
- \mathbb{E}\big[e^{-\rho\tau_1}X_{\tau_1}^1\mathbb{I}_{\{\tau_1<\infty\}}\big]+\mathbb{E}\big[e^{-\rho\tau_1}X_{\tau_1}^2\mathbb{I}_{\{\tau_1<\infty\}}\big]\\
\le &~ x_1-\mathbb{E}\left[x_2 \mathbb{I}_{\{\tau_2<\infty\}}\right]+\mathbb{E}\left[\int_0^{\tau_2}\left(\rho-\mu_2\right)e^{-\rho t}X_t^2\dt \mathbb{I}_{\{\tau_2<\infty\}}\right]\\
&+x_2-\mathbb{E}\left[x_1 \mathbb{I}_{\{\tau_1<\infty\}}\right]+\mathbb{E}\left[\int_0^{\tau_1}\left(\rho-\mu_1\right)e^{-\rho t}X_t^1\dt \mathbb{I}_{\{\tau_1<\infty\}}\right]
\end{align*}
Now, 
\begin{align*}
\mathbb{E}\left[\int_0^{\tau_1}\left(\rho-\mu_1\right)e^{-\rho t}X_t^1\dt \mathbb{I}_{\{\tau_1<\infty\}}\right]
&\le   \mathbb{E}\left[\int_0^{\infty}\left(\rho-\mu_1\right)e^{-\rho t}X_t^1\dt \right]\\
&=  \left(\rho-\mu_1\right)\int_0^{\infty}e^{-\rho t}x_1e^{\mu_1 t}\dt \\
& =  x_1.
\end{align*}
Similarly,
\begin{align*}
\mathbb{E}\left[\int_0^{\tau_2}\left(\rho-\mu_2\right)e^{-\rho t}X_t^2\dt \mathbb{I}_{\{\tau_2<\infty\}}\right] & \le x_2
\end{align*}
Thus, $J_0(x_1,x_2,\Lambda_0)\le 2x_1+2x_2.$
\end{proof}
\section{HJB Equations}
In this section, we study the associated HJB equations.  To the above stochastic differential equation (\ref{SDE}), we assign the following partial differential operator
\begin{equation}\label{pdo}
\mathcal{A}=\frac{1}{2}\left\{a_{11}x_1^2\frac{\partial^2}{\partial x_1^2}
+2a_{12}x_1x_2\frac{\partial^2}{\partial x_1 \partial x_2}
+a_{22}x_2^2\frac{\partial^2}{\partial x_2^2}\right\}
+\mu_1 x_1\frac{\partial}{\partial x_1}
+\mu_2 x_2\frac{\partial}{\partial x_2}, 
\end{equation}
where
$a_{11}=\sigma^2_{11}+\sigma^2_{12},\
a_{12}=\sigma_{11}\sigma_{21}+\sigma_{12}\sigma_{22},\mbox{ and }
a_{22}=\sigma^2_{21}+\sigma^2_{22}$ \cite{Oks}.
The associated HJB equations have the form:
For $x_1,x_2>0$,
\beq{HJB}
\begin{cases}
\min\Big\{\rho v_0(x_1,x_2)-\A v_0(x_1,x_2),\
v_0(x_1,x_2)-v_1(x_1,x_2)+\Bb  x_1-\Bs  x_2\Big\}=0,\\
\min\Big\{\rho v_1(x_1,x_2)-\A v_1(x_1,x_2),\
v_1(x_1,x_2)-\Bs  x_1+\Bb  x_2\Big\}=0.\\
\end{cases}
\eeq
\par To solve the above HJB equations, we first convert them into single variable
equations. Let $y=x_2/x_1$ and $v_i(x_1,x_2)=x_1w_i(x_2/x_1)$,
for some function $w_i(y)$ and $i=0,1$.  Then write $\fA v_i$ in terms of $w_i$ to obtain
\begin{align*}
\fA v_i&=\frac12\left\{a_{11}{x_1}^2\left(\frac{y^2w''_i(y)}{x_1}\right)+2a_{12}x_1x_2\left(-\frac{yw''_i(y)}{x_1}\right)+a_{22}{x_2}^2\left(\frac{w''_i(y)}{x_1}\right)\right\}\\
&\indent +\mu_1x_1\left(w_i(y)-yw'_i(y)\right)+\mu_2x_2\left(w'_i(y)\right)\\
&=x_1\left\{\frac{1}{2}\left[a_{11}-2a_{12}+a_{22}\right]y^2w''_i(y)+(\mu_2-\mu_1)y w'_i(y)+\mu_1w_i(y)\right\}.
\end{align*}
Let $\displaystyle{\mathcal{L} w_i(y)=\lambda y^2w''_i(y)+(\mu_2-\mu_1)yw'_i(y)+\mu_1w_i(y)}$, where $\displaystyle{\lambda=\frac{a_{11}-2a_{12}+a_{22}}{2}}$.  \\
So $\fA v_i=x_1\mathcal{L} w_i$.  Note that $\lambda \ge 0$ since 
\begin{align*}
\lambda
&=\frac12\left[(\sigma_{11}-\sigma_{21})^2+(\sigma_{12}-\sigma_{22})^2\right].
\end{align*}
Here we only consider the case when $\lambda\not=0$.
If $\lambda=0$, the problem reduces to a first order case and
can be treated accordingly.
The HJB equations can be given in terms of $y$ and $w_i$ as follows:
\beq{HJB2}
\begin{cases}
\min\Big\{(\rho-\mathcal{L}) w_0(y),\
w_0(y)-w_1(y)+\Bb  -\Bs  y\Big\}=0,\\
\min\Big\{(\rho-\mathcal{L}) w_1(y),\
w_1(y)-\Bs  +\Bb  y\Big\}=0.\\
\end{cases}
\eeq
\par To solve the above HJB equations, we first consider the equations 
$(\rho-\mathcal{L})w_i(y)=0$, $i=0,1$, which can be rewritten as
$$
-\lambda y^2w''_i(y)-(\mu_2-\mu_1)yw'_i(y)+(\rho-\mu_1)w_i(y)=0.\\
$$
Clearly, these are the Euler equations and their solutions 
are of the form $y^{\delta}$, for some $\delta$.
Substitute this into the equation $(\rho-\mathcal{L}) w_i=0$ to obtain
$$
\delta^2-\left(1+\frac{\mu_1-\mu_2}{\lambda}\right) \delta-\frac{\rho-\mu_1}{\lambda}=0.
$$
This equation has two roots, $\delta_1$ and $\delta_2$, given by
\begin{align}\label{delta1}
\delta_1&=\frac{1}{2}\Biggl(1+\frac{\mu_1-\mu_2}{\lambda}+\sqrt{\(1+\frac{\mu_1-\mu_2}{\lambda}\)^2+\frac{4\rho-4\mu_1}{\lambda}}\,\Biggr),
\end{align}
\begin{align}\label{delta2}
\delta_2&=\frac{1}{2}\Biggl(1+\frac{\mu_1-\mu_2}{\lambda}-
\sqrt{\(1+\frac{\mu_1-\mu_2}{\lambda}\)^2+\frac{4\rho-4\mu_1}{\lambda}}\,\Biggr).
\end{align}
These roots are both real since we assume $\rho>\mu_1$.  We also assume $\rho>\mu_2$, and together these assumptions lead to $\delta_1>1$ and $\delta_2<0$.

We conclude that the general solution of $(\rho-\mathcal{L})w_i(y)=0$ should be of the form:
$w_i(y)=c_{i1}y^{\delta_1}+c_{i2}y^{\delta_2}$,
for some constants $c_{i1}$ and $c_{i2}$, $i=0,1$.  
Note that as $y\to 0$, $y^{\delta_2}\to \infty$, and as $y \to \infty$, $y^{\delta_1}\to\infty$. 
Also note the following identities in $\delta_1$ and $\delta_2$:
\begin{align}
&-\delta_1\delta_2
=\frac{\rho-\mu_1}{\lambda},\\
&\delta_1+\delta_2
=1+\frac{\mu_1-\mu_2}{\lambda},\\
&(\delta_1-1)(1-\delta_2)
=\frac{\rho-\mu_2}{\lambda},\\
&\frac{-\delta_1\delta_2}{(\delta_1-1)(1-\delta_2)}
=\frac{\rho-\mu_1}{\rho-\mu_2}.
\end{align}
Now, the second part of the HJB equation 
\[\min\Big\{(\rho-\mathcal{L}) w_1(y),\
w_1(y)-\Bs  +\Bb  y\Big\}=0\]
is independent of $w_0$ and can be solved first.  We must find thresholds $k_1$ and $k_2$ for buying and selling, as in \cite{JOptim}.

\begin{figure}
\includegraphics[width=0.75\textwidth]{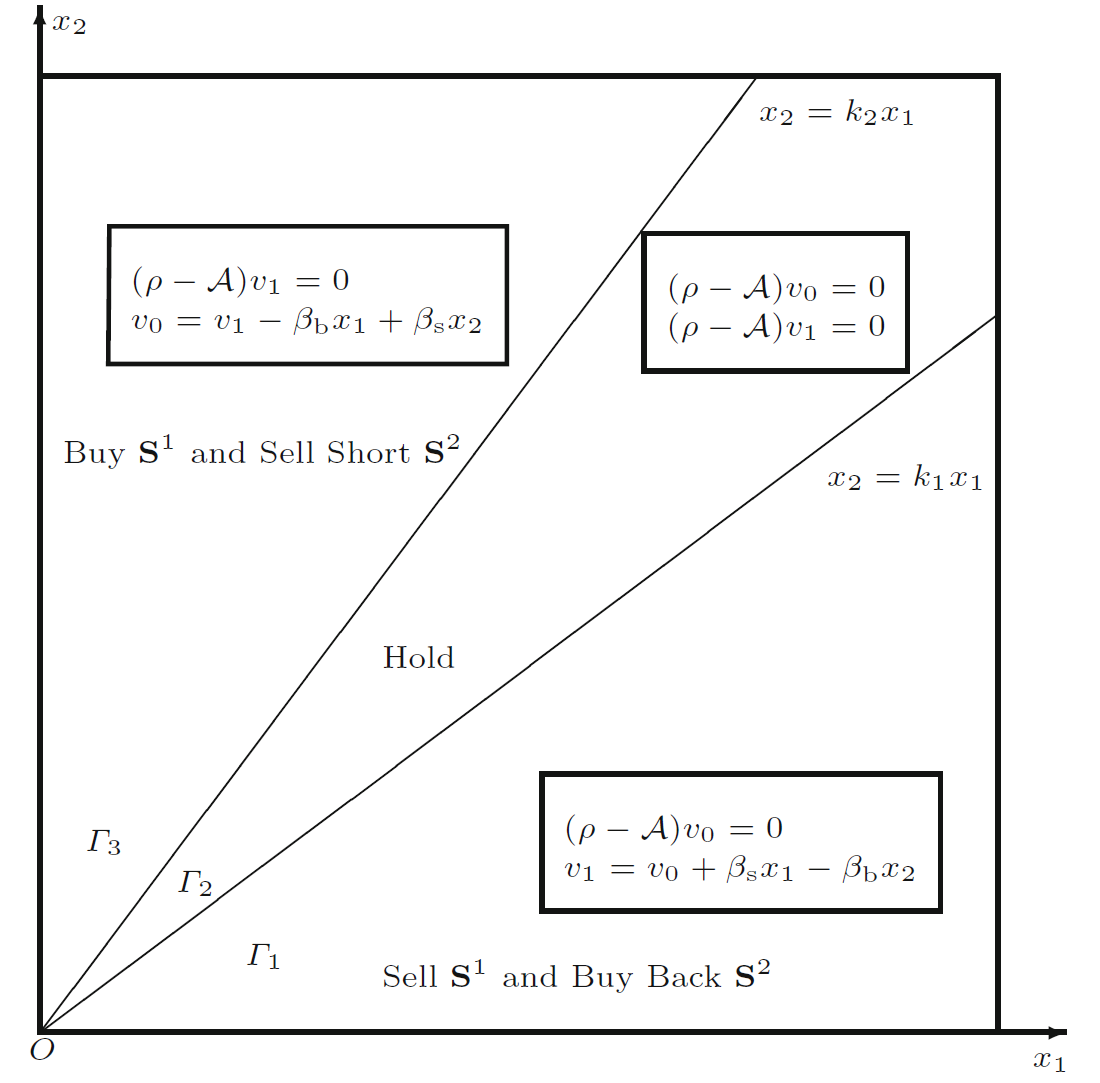}
\caption{Thresholds for buying and selling regions}
\end{figure}

 First, we need to find $k_1$ so that 
on the interval $[0,k_1]$, $w_1(y)=\Bs-\Bb y$, and on the interval $[k_1, \infty)$,
$w_1(y)=C_2y^{\delta_2}$. 
Then the smooth-fit conditions determine $k_1$ and $C_2$.
Necessarily, the continuity of $w_1$ and its first order derivative at 
$y=k_1 $ imply
\[
\Bs-\Bb k_1 =C_2k_1^{\delta_2}\quad \text{and}\quad
-\Bb =C_2\delta_2 k_1^{\delta_2-1}.\]
From this system of equations, we can see
\begin{align}
& k_1=\frac{\Bs}{\Bb}\cdot\frac{-\delta_2}{1-\delta_2}.
\end{align}
and 
\begin{align}
&C_2=\frac{\Bb}{-\delta_2}\left(\frac{\Bs}{\Bb}\cdot\frac{-\delta_2}{1-\delta_2}\right)^{1-\delta_2}
=\left(\frac{\Bs}{1-\delta_2}\right)^{1-\delta_2}\left(\frac{\Bb}{-\delta_2}\right)^{\delta_2}.
\end{align}
We obtain the function
\beq{}
w_1(y)=\begin{cases} \Bs-\Bb y, \/\ \/\ &  \text{if   }\/\  y<k_1\\
C_2 y^{\delta_2},\/\ \/\ & \text{if   }\/\ y\geq k_1
\end{cases}
\eeq
with $k_1$ and $C_2$ given above. Next we need to solve the first part of HJB equation:
\[\min\Big\{(\rho -\mathcal{L}) w_0(y),\
w_0(y)-w_1(y)+\Bb  -\Bs  y\Big\}=0.\]
We need to find $k_2$ so that on the interval $[0,k_2]$,
$w_0(y)=C_1y^{\delta_1}$, and on the interval $[k_2,\infty),\\
w_0(y)=w_1(y)-\Bb+\Bs y=C_2y^{\delta_2}-\Bb+\Bs y$.
Then the continuity of $w_0$ and its first order derivative at $y=k_2$ yield
\[C_1k_2^{\delta_1}= C_2k_2^{\delta_2}-\Bb+\Bs k_2 \quad\text{and}\quad
C_1\delta_1 k_2^{\delta_1-1}= C_2\delta_2 k_2^{\delta_2-1}+\Bs.\]
Take the ratio of the above two equations and get
\[\frac{k_2}{\delta_1}=\frac{C_2k_2^{\delta_2}-\Bb+\Bs k_2}{C_2\delta_2 k_2^{\delta_2-1}+\Bs}.\]
This implies 
\begin{align*}
C_2(\delta_1-\delta_2)k_2^{\delta_2}+\Bs(\delta_1-1)k_2-\Bb\delta_1=0.
\end{align*}
This is an equation of $k_2$:
\beq[
f(k_2):=C_2(\delta_1-\delta_2)k_2^{\delta_2}+\Bs (\delta_1-1)k_2-\Bb\delta_1 =0.
\eeq
Consider 
\[f(y):=C_2(\delta_1-\delta_2)y^{\delta_2}+\Bs (\delta_1-1)y-\Bb\delta_1. \]
Note that as $y\to\infty$, $f(y)\to\Bs(\delta_1-1)y-\Bb\delta_1$, since $\delta_2<0$.  That is, as $y\to\infty$, $f(y)\to\infty$, since $\Bs>0$, $\delta_1-1>0$.  Also, as $y\to 0^+$,
 $f(y)\to C_2(\delta_1-\delta_2)y^{\delta_2}-\Bb\delta_1$.  That is, as $y\to 0^+$, $f(y)\to\infty$, since $C_2>0$, $\delta_1-\delta_2>0$, and $\delta_2<0$.  Now, 
\begin{align*}
&\indent f'(y)=C_2\delta_2(\delta_1-\delta_2)y^{\delta_2-1}+\Bs(\delta_1-1)\\
&\indent f''(y)=C_2\delta_2(\delta_2-1)(\delta_1-\delta_2)y^{\delta_2-2}=C_2(-\delta_2)(1-\delta_2)(\delta_1-\delta_2)y^{\delta_2-2}.
\end{align*}
Note then that $f''(y)>0$ for all $y>0$ since $C_2>0$, $(-\delta_2)>0$, $(1-\delta_2)>0$, and $(\delta_1-\delta_2)>0$.
Hence $f$ is convex for all $y>0$.  Also note that
\begin{align*}
f'(y)=0
&\iff y=\left[\frac{\Bs(\delta_1-1)}{C_2(-\delta_2)(\delta_1-\delta_2)}\right]^{\frac1{\delta_2-1}}.
\end{align*}
Hence $f$ attains its global minimum at $\displaystyle{y_c=\left[\frac{\Bs(\delta_1-1)}{C_2(-\delta_2)(\delta_1-\delta_2)}\right]^{\frac1{\delta_2-1}}}>0$. We will show that $f(y)=0$ has two solutions and take the larger one to be $k_2$. Since we already know $C_2$, once we find $k_2$, we can express $C_1$ using the relationship above:
\begin{align}
C_1= \frac{C_2\delta_2 k_2^{\delta_2-1}+\Bs}{\delta_1 k_2^{\delta_1-1}}&=\left(\frac{\Bs}{\Bb}\cdot\frac{-\delta_2}{1-\delta_2}\right)^{1-\delta_2}\frac{\Bb}{-\delta_2}\frac{\delta_2k_2^{\delta_2-1}}{\delta_1k_2^{\delta_1-1}}+\frac{\Bs}{\delta_1k_2^{\delta_1-1}}\\
&=\left[1-\left(\frac{\Bs}{\Bb}\right)^{-\delta_2}\left(\frac{-\delta_2}{1-\delta_2}\right)^{1-\delta_2}k_2^{\delta_2-1}\right]\left(\frac{\Bs}{\delta_1}\right)k_2^{1-\delta_1}\notag.
\end{align}
We show that $f(y_c)<0$, thus implying the existence of $k_2$. We compute $f(y_c)$:
\begin{align*}
f(y_c) &=C_2(\delta_1-\delta_2)y_c^{\delta_2}+\Bs (\delta_1-1)y_c-\Bb\delta_1\\
&=C_2^{\frac{1}{1-\delta_2}} (\delta_1-\delta_2)^{\frac{1}{1-\delta_2}}[\Bs (\delta_1-1)]^{\frac{\delta_2}{\delta_2-1}}[ (-\delta_2)^{-\frac{\delta_2}{\delta_2-1}}
+(-\delta_2)^{\frac{1}{1-\delta_2}}]-\Bb\delta_1.
\end{align*}
Next we insert $\displaystyle{C_2=\left(\frac{\Bs}{1-\delta_2}\right)^{1-\delta_2} \cdot \left(\frac{\Bb}{-\delta_2}\right)^{\delta_2}}$ into $f(y_c)$ to get
\begin{align*}
f(y_c) &= \left(\frac{\Bs}{1-\delta_2}\right) \left(\frac{\Bb}{-\delta_2}\right)^{\frac{\delta_2}{1-\delta_2}} (\delta_1-\delta_2)^{\frac{1}{1-\delta_2}}[\Bs (\delta_1-1)]^{\frac{\delta_2}{\delta_2-1}}[ (-\delta_2)^{-\frac{\delta_2}{\delta_2-1}}+(-\delta_2)^{\frac{1}{1-\delta_2}}]-
\Bb\delta_1\\
&=\Bb \left[ \left(\frac{\Bs}{\Bb}\right)^{1+\frac{-\delta_2}{1-\delta_2}} (\delta_1-\delta_2)^{\frac{1}{1-\delta_2}} (\delta_1-1)^{\frac{\delta_2}{\delta_2-1}}
-\delta_1\right].
\end{align*}
Since $\delta_2<0$, we let $\delta_2=-r$ with $r>0$ and $\displaystyle{\beta=\frac{\beta_b}{\beta_s}>1}$. This will imply
\begin{align*}
f(y_c) &=\Bb \left[ \left(\frac{\Bs}{\Bb}\right)^{1+\frac{r}{1+r}} (\delta_1+r)^{\frac{1}{1+r}} (\delta_1-1)^{\frac{r}{1+r}}-\delta_1\right]\\
&=\Bb \delta_1\left[ \beta^{-1-\frac{r}{1+r}} \left(1+\frac{r}{\delta_1}\right)^{\frac{1}{1+r}} \left(1-\frac{1}{\delta_1}\right)^{\frac{r}{1+r}}
-1\right].
\end{align*}
The necessary and sufficient condition for the existence of $k_2$ is $f(y_c) \leq 0$, and this is equivalent to
\[  \left(1+\frac{r}{\delta_1}\right)^{\frac{1}{1+r}} \left(1-\frac{1}{\delta_1}\right)^{\frac{r}{1+r}}
\leq \beta^{\frac{1+2r}{1+r}}. \]
We apply the geometric-arithmetic mean inequality
\[ A^{\theta}B^{1-\theta}\leq \theta A+(1-\theta)B\ \text{with}\ \theta=\frac{1}{1+r},\ A=1+\frac{r}{\delta_1}\ \text{and}\ B=1-\frac{1}{\delta_1}\]
to the left hand side of the above inequality to get
\[\left(1+\frac{r}{\delta_1}\right)^{\frac{1}{1+r}} \left(1-\frac{1}{\delta_1}\right)^{\frac{r}{1+r}}\leq \left(1+\frac{r}{\delta_1}\right) \cdot {\frac{1}{1+r}} +
 \left(1-\frac{1}{\delta_1}\right)\cdot {\frac{r}{1+r}} =1.\]
This implies $f(y_c)\leq 0$ if 
\[ 1<\beta^{\frac{1+2r}{1+r}} \quad  \Longleftrightarrow \quad
1 <\beta.\]
 This obviously holds  since $\beta>1$.  So we establish the existence of $k_2$. 

Note that it is clear that $C_2>0$.  We also wish to establish $C_1>0$.  Consider,
\begin{align*}
C_1>0 &\iff C_2\delta_2k_2^{\delta_2-1}+\Bs>0\\
&\iff \left(\frac{\Bs}{\Bb}\cdot\frac{-\delta_2}{1-\delta_2}\right)^{1-\delta_2}\frac{\Bb}{-\delta_2}\delta_2 k_2^{\delta_2-1}+\Bs>0\\
&\iff k_2> \left(\frac{\Bs}{\Bb}\right)^{\frac{-\delta_2}{1-\delta_2}}\left(\frac{-\delta_2}{1-\delta_2}\right).
\end{align*}
Note then that if $\displaystyle{f\left(\left(\frac{\Bs}{\Bb}\right)^{\frac{-\delta_2}{1-\delta_2}}\left(\frac{-\delta_2}{1-\delta_2}\right)\right)}<0$, we establish $\displaystyle{C_1>0}$.\\
$f\left(\left(\dfrac{\Bs}{\Bb}\right)^{\frac{-\delta_2}{1-\delta_2}}\left(\dfrac{-\delta_2}{1-\delta_2}\right)\right)$
\begin{align*}
&=C_2(\delta_1-\delta_2)\left[\left(\frac{\Bs}{\Bb}\right)^{\frac{-\delta_2}{1-\delta_2}}\left(\frac{-\delta_2}{1-\delta_2}\right)\right]^{\delta_2}
+\Bs(\delta_1-1)\left(\frac{\Bs}{\Bb}\right)^{\frac{-\delta_2}{1-\delta_2}}\left(\frac{-\delta_2}{1-\delta_2}\right)-\Bb\delta_1\\
&=\Bb\delta_1\left[\left(\frac{\Bs}{\Bb}\right)^{1+\frac{-\delta_2}{1-\delta_2}}-1\right]\\
&<0
\end{align*}
since $\displaystyle{\left(\frac{\Bs}{\Bb}\right)^{1+\frac{-\delta_2}{1-\delta_2}}<\frac{\Bs}{\Bb}<1}$.  Hence we have shown that $C_1>0$.
Now we consider the following regions:
\[
\begin{array}{l}
\disp
\Gamma_1=(0,k_1]\\
\disp
\Gamma_2=(k_1,k_2)\\
\disp
\Gamma_3=[k_2,\infty)\\
\end{array}
\]
We have chosen $k_1$, $k_2$ such that we establish the following equalities
\[
\begin{array}{l}
\disp
\Gamma_1: \/\ \/\ w_1(y)-\Bs+\Bb y= 0 \/\ , \\
\/\  (\rho-\mathcal{L})w_0(y)=0\\
\disp
\\
\Gamma_2:\/\ \/\ (\rho-\mathcal{L})w_1(y)=0 \/\ ,\\
 \/\ (\rho-\mathcal{L})w_0(y)=0\\
\disp
\\
\Gamma_3:\/\ \/\ (\rho-\mathcal{L})w_1(y)=0 \/\ , \\
\/\ w_0(y)-w_1(y)+\Bb-\Bs y=0\\
\end{array}
\]
for solutions of the form 
\[ w_0(y)=\begin{cases} 
      C_1y^{\delta_1} & y\in\Gamma_1 \\
      C_1y^{\delta_1} & y\in\Gamma_2 \\
      C_2y^{\delta_2}-\Bb+\Bs y & y\in\Gamma_3
   \end{cases},
\]
\[ w_1(y)=\begin{cases} 
      \Bs-\Bb y & y\in\Gamma_1 \\
      C_2y^{\delta_2} & y\in\Gamma_2 \\
      C_2y^{\delta_2} & y\in\Gamma_3
   \end{cases}.
\]
We now proceed to establish the following variational inequalities, thus confirming that we have solved the HJB equation:
\[
\begin{array}{l}
\disp
\Gamma_1: \/\ \/\ (\rho-\mathcal{L})w_1(y) \ge 0,\\
\disp
 w_0(y)-w_1(y)+\Bb-\Bs y\ge 0\\
\\
\disp
\Gamma_2:\/\ \/\ w_1(y)-\Bs+\Bb y\ge 0 ,  \\
\disp
 w_0(y)-w_1(y)+\Bb-\Bs y\ge 0  \\
\\
\disp
\Gamma_3:\/\ \/\  w_1(y)-\Bs+\Bb y \ge 0 , \\
\disp
   (\rho-\mathcal{L})w_0(y)\ge 0  .  \\
\end{array}
\]
\\
\textbf{\underline{$y\in \Gamma_1$:}} Using $(\rho-\mathcal{L})w_0(y)=0$ and $w_1(y)=\Bs-\Bb y$, we obtain
\begin{align*}
w_0(y)-w_1(y)+\Bb-\Bs y &=C_1y^{\delta_1}-\Bs+\Bb y+\Bb-\Bs y\\
&=C_1y^{\delta_1}+(\Bb-\Bs)(y+1)\\
&\ge 0
\end{align*}
since $C_1>0$, $\Bb>\Bs$, and $y>0$.  \\
Also,
\begin{align*}
(\rho-\mathcal{L})w_1(y) & = (\rho-\L)(\Bs-\Bb y)\\
&=(\rho-\mu_1)\Bs-(\rho-\mu_2)\Bb y\\
\implies (\rho-\mathcal{L})w_1(y)\ge 0 &\iff (\rho-\mu_1)\Bs-(\rho-\mu_2)\Bb y\ge 0\\
&\iff \frac{(\rho-\mu_1)\Bs}{(\rho-\mu_2)\Bb}\ge y\\
&\iff \frac{(\rho-\mu_1)\Bs}{(\rho-\mu_2)\Bb}\ge k_1
\end{align*}
since $k_1\ge y$ for all $y\in\Gamma_1$. But note that
\begin{align*}
\frac{(\rho-\mu_1)\Bs}{(\rho-\mu_2)\Bb}\ge k_1&\iff \frac{-\delta_1\delta_2}{(\delta_1-1)(1-\delta_2)}\cdot\frac{\Bs}{\Bb}\ge k_1\\
&\iff \frac{\delta_1}{(\delta_1-1)}\cdot k_1\ge k_1,
\end{align*}
which obviously holds since $\delta_1>\delta_1-1>0$.  Thus we have established the variational inequalities for the region $\Gamma_1$.\\
\\
\textbf{\underline{$y\in \Gamma_3$:}} Using $(\rho-\mathcal{L})w_1(y)=0$ and $w_1(y)=w_0(y)+\Bb-\Bs y$, we obtain
\begin{align*}
w_1(y)-\Bs+\Bb y &=w_0(y)+\Bb-\Bs y-\Bs+\Bb y\\
&=C_2y^{\delta_2}-\Bb+\Bs y+\Bb-\Bs y-\Bs+\Bb y\\
&=C_2y^{\delta_2}+\Bb y-\Bs
\end{align*}
Note that the continuity of $w_1$ and $w'_1$ at $k_1$ ensure that 
\[
\begin{array}{l}
\disp C_2k_1^{\delta_2}+\Bb k_1-\Bs=0\\
\disp C_2\delta_2k_1^{\delta_2-1}+\Bb=0.\\
\end{array}
\]
Let $g(y)=C_2y^{\delta_2}+\Bb y-\Bs$.  Then $g'(y)=C_2\delta_2y^{\delta_2-1}+\Bb$.  Note that
\begin{align*}
g'(y)\ge 0\iff C_2\delta_2y^{\delta_2-1}+\Bb\ge0 
&\iff \frac{C_2(-\delta_2)}{\Bb}\le y^{1-\delta_2}\\
&\iff k_1^{1-\delta_2}\le y^{1-\delta_2}\\
&\iff k_1\le y.
\end{align*}
Thus $g(y)=C_2y^{\delta_2}+\Bb y-\Bs$ is increasing for all $y\ge k_1$.  In particular, since $C_2k_1^{\delta_2}+\Bb k_1-\Bs=0$, we must have $C_2y^{\delta_2}+\Bb y-\Bs\ge 0$ for all $y\ge k_1$.  Thus $C_2y^{\delta_2}+\Bb y-\Bs=w_1(y)-\Bs+\Bb y\ge 0$ for all $y\in \Gamma_2\cup\Gamma_3$.\\
Also,
\begin{align*}
(\rho-\mathcal{L})w_0(y)
&=(\rho-\mathcal{L})(w_1(y))+(\rho-\mathcal{L})(\Bs y-\Bb)\\
&=\rho\Bs y-\rho\Bb+\mu_1\Bb-\mu_2\Bs y\\
&=(\rho-\mu_2)\Bs y-(\rho-\mu_1)\Bb.
\end{align*}
Hence 
\begin{align*}
(\rho-\mathcal{L})w_0(y)\ge 0&\iff (\rho-\mu_2)\Bs y-(\rho-\mu_1)\Bb\ge 0\\
&\iff y\ge \frac{(\rho-\mu_1)\Bb}{(\rho-\mu_2)\Bs}\\
&\iff k_2\ge\frac{(\rho-\mu_1)\Bb}{(\rho-\mu_2)\Bs}
\end{align*}
since $k_2\le y$ for all $y\in\Gamma_3$.  Note that $\displaystyle{\frac{(\rho-\mu_1)\Bb}{(\rho-\mu_2)\Bs}=\frac{-\delta_1\delta_2}{(\delta_1-1)(1-\delta_2)}\cdot\frac{\Bb}{\Bs}}$ and consider\\
$\displaystyle f\left(\frac{-\delta_1\delta_2}{(\delta_1-1)(1-\delta_2)}\cdot\frac{\Bb}{\Bs}\right)$
\begin{align*}
&=C_2(\delta_1-\delta_2)\left(\frac{-\delta_1\delta_2}{(\delta_1-1)(1-\delta_2)}\cdot\frac{\Bb}{\Bs}\right)^{\delta_2}
+\Bs(\delta_1-1)\left(\frac{-\delta_1\delta_2}{(\delta_1-1)(1-\delta_2)}\cdot\frac{\Bb}{\Bs}\right)-\Bb\delta_1\\
&=\frac{\delta_1-\delta_2}{1-\delta_2}\left(\frac{\delta_1}{\delta_1-1}\right)^{\delta_2}\beta^{2\delta_2-1}\Bb+\Bb\delta_1\left(\frac{-\delta_2}{1-\delta_2}-1\right).
\end{align*}
Now, let $\delta_2=-r$ with $r>0$.  Then
\begin{align*}
f\left(\frac{-\delta_1\delta_2}{(\delta_1-1)(1-\delta_2)}\cdot\frac{\Bb}{\Bs}\right)&=\left(\frac{\delta_1+r}{1+r}\right)\left(\frac{\delta_1-1}{\delta_1}\right)^r\beta^{-2r-1}\Bb+\Bb\delta_1\left(\frac{r}{1+r}-1\right).
\end{align*}
Hence
\begin{align*}
f\left(\frac{-\delta_1\delta_2}{(\delta_1-1)(1-\delta_2)}\cdot\frac{\Bb}{\Bs}\right)<0&\iff\left(\frac{\delta_1+r}{1+r}\right)\left(\frac{\delta_1-1}{\delta_1}\right)^r\beta^{-2r-1}\Bb<\Bb\delta_1\left(\frac{-r+1+r}{1+r}\right)\\
&\iff\left(1+\frac{r}{\delta_1}\right)^{\frac{1}{r+1}}\left(1-\frac{1}{\delta_1}\right)^{\frac{r}{r+1}}<\beta^{\frac{2r+1}{r+1}}.
\end{align*}
Applying the arithmetic-geometric mean inequality to the left-hand side yields
\begin{align*}
\left(1+\frac{r}{\delta_1}\right)^{\frac{1}{r+1}}\left(1-\frac{1}{\delta_1}\right)^{\frac{r}{r+1}}&\le\left(\frac{1}{r+1}\right)\left(1+\frac{r}{\delta_1}\right)+\left(\frac{r}{r+1}\right)\left(1-\frac{1}{\delta_1}\right)\\
&=\frac{1}{r+1}+\frac{r}{r+1}\cdot\frac{1}{\delta_1}+\frac{r}{r+1}-\frac{r}{r+1}\cdot\frac{1}{\delta+1}\\
&=\frac{r+1}{r+1}=1<\beta<\beta^{\frac{2r+1}{r+1}}.
\end{align*}
So $\displaystyle{f\left(\frac{-\delta_1\delta_2}{(\delta_1-1)(1-\delta_2)}\cdot\frac{\Bb}{\Bs}\right)<0}$ holds.  That is, $\displaystyle{k_2>\frac{(\rho-\mu_1)}{(\rho-\mu_2)}\cdot\frac{\Bb}{\Bs}}$, which establishes $(\rho-\mathcal{L})w_0(y)\ge0$ for all $y\in\Gamma_3$.\\
\\
\textbf{\underline{$y\in \Gamma_2$:}}  On $\Gamma_2$, we have $w_1(y)-\Bs+\Bb y=C_2y^\delta_2-\Bs+\Bb y$.  Note that we have already shown that $C_2y^\delta_2-\Bs+\Bb y\ge 0$ for all $y\in\Gamma_2\cup\Gamma_3$.  Hence, $w_1(y)-\Bs+\Bb y\ge 0$ for all $y\in\Gamma_2$.\\
We also have $w_0(y)-w_1(y)+\Bb-\Bs y=C_1y^\delta_1-C_2y^\delta_2+\Bb-\Bs y$.  Let 
\[
\begin{array}{l}
\disp \phi(y)=C_1y^{\delta_1}-C_2y^{\delta_2}+\Bb-\Bs y.\\
\end{array}
\]
Hence
\[
\begin{array}{l}
\disp \phi'(y)=C_1\delta_1y^{\delta_1-1}+C_2(-\delta_2)y^{\delta_2-1}-\Bs \\
\disp \phi''(y)=C_1\delta_1(\delta_1-1)y^{\delta_1-2}-C_2(-\delta_2)(1-\delta_2)y^{\delta_2-2}.\\
\end{array}
\]
By continuity of $w_0$, we know $C_1 k_2^{\delta_1}-C_2 k_2^{\delta_2}+\Bb-\Bs k_2=0$.  That is, we know $\phi(k_2)=0$.\\
By continuity of $w'_0$, we know $C_1\delta_1 k_2^{\delta_1-1}+C_2 (-\delta_2) k_2^{\delta_2-1}-\Bs=0$.  That is, we know $\phi'(k_2)=0$.\\
By continuity of $w_1$, we know $C_2 k_1^{\delta_2}=\Bs-\Bb k_1$.  Hence, $C_1k_1^{\delta_1}-C_2k_1^{\delta_2}+\Bb-\Bs k_1=C_1 k_1^{\delta_1}-\Bs+\Bb k_1+\Bb-\Bs k_1=C_1 k_1^{\delta_1}+(k_1+1)(\Bb-\Bs)\ge0$.  That is, we know $\phi(k_1)\ge0$.\\
By continuity of $w'_1$, we know $-C_2(-\delta_2)k_1^{\delta_2-1}=-\Bb$.  Hence, $C_1\delta_1k_1^{\delta_1-1}+C_2(-\delta_2)k_1^{\delta_2-1}-\Bs =C_1\delta_1k_1^{\delta_1-1}+\Bb-\Bs\ge0$.  That is, we know $\phi'(k_1)\ge0$.  Now,
\begin{align*}
\phi''(y)&=C_1\delta_1(\delta_1-1)y^{\delta_1-2}-C_2(-\delta_2)(1-\delta_2)y^{\delta_2-2}\\
&=\left(\frac{C_2\delta_2k_2^{\delta_2-1}+\Bs}{\delta_1k_2^{\delta_1-1}}\right)\delta_1(\delta_1-1)y^{\delta_1-2}-C_2(-\delta_2)(1-\delta_2)y^{\delta_2-2}\\
&=-C_2(-\delta_2)k_2^{\delta_2-2}\left[(\delta_1-1)\left(\frac{y}{k_2}\right)^{\delta_1-2}+(1-\delta_2)\left(\frac{y}{k_2}\right)^{\delta_2-2}\right]+\Bs(\delta_1-1)k_2^{-1}\left(\frac{y}{k_2}\right)^{\delta_1-2}.\\
\end{align*}
Hence $\phi''(k_2)=\Bs(\delta_1-1)k_2^{-1}-C_2(-\delta_2)(\delta_1-\delta_2)k_2^{\delta_2-2}$.  Then note that
\begin{align*}
k_2>\left[\frac{\Bs(\delta_1-1)}{C_2(\delta_1-\delta_2)(-\delta_2)}\right]^{\frac{1}{\delta_2-1}}&\implies k_2^{\delta_2-1}<\frac{\Bs(\delta_1-1)}{C_2(\delta_1-\delta_2)(-\delta_2)}
\end{align*}
since $\delta_2-1<0$.  Thus,
\begin{align*}
&(k_2^{\delta_2-1})k_2^{-1}(-C_2)(-\delta_2)(\delta_1-\delta_2)>\left(\frac{\Bs(\delta_1-1)}{C_2(\delta_1-\delta_2)(-\delta_2)}\right)k_2^{-1}(-C_2)(-\delta_2)(\delta_1-\delta_2)\\
&\implies (k_2^{\delta_2-2})(-C_2)(-\delta_2)(\delta_1-\delta_2)>-\Bs(\delta_1-1)k_2^{-1}\\
&\implies \Bs(\delta_1-1)k_2^{-1}-C_2(-\delta_2)(\delta_1-\delta_2)k_2^{\delta_2-2}>0
\end{align*}
That is, $\phi''(k_2)>0$.\\
Consider the equation $\phi''(y)=0$.
\begin{align*}
\phi''(y)=0&\iff C_1\delta_1(\delta_1-1)y^{\delta_1-2}-C_2(-\delta_2)(1-\delta_2)y^{\delta_2-2}=0\\
&\iff C_1\delta_1(\delta_1-1)y^{\delta_1-2}=C_2(-\delta_2)(1-\delta_2)y^{\delta_2-2}\\
&\iff y^{\delta_1-\delta_2}=\frac{C_2(-\delta_2)(1-\delta_2)}{C_1\delta_1(\delta_1-1)}\\
&\iff y=\left(\frac{C_2(-\delta_2)(1-\delta_2)}{C_1\delta_1(\delta_1-1)}\right)^{\frac{1}{\delta_1-\delta_2}}
\end{align*}
Note then that $\phi''(y)=0$ has a unique solution in $[k_1,k_2]$.\\

\begin{figure}[ht]
\includegraphics[width=0.75\textwidth]{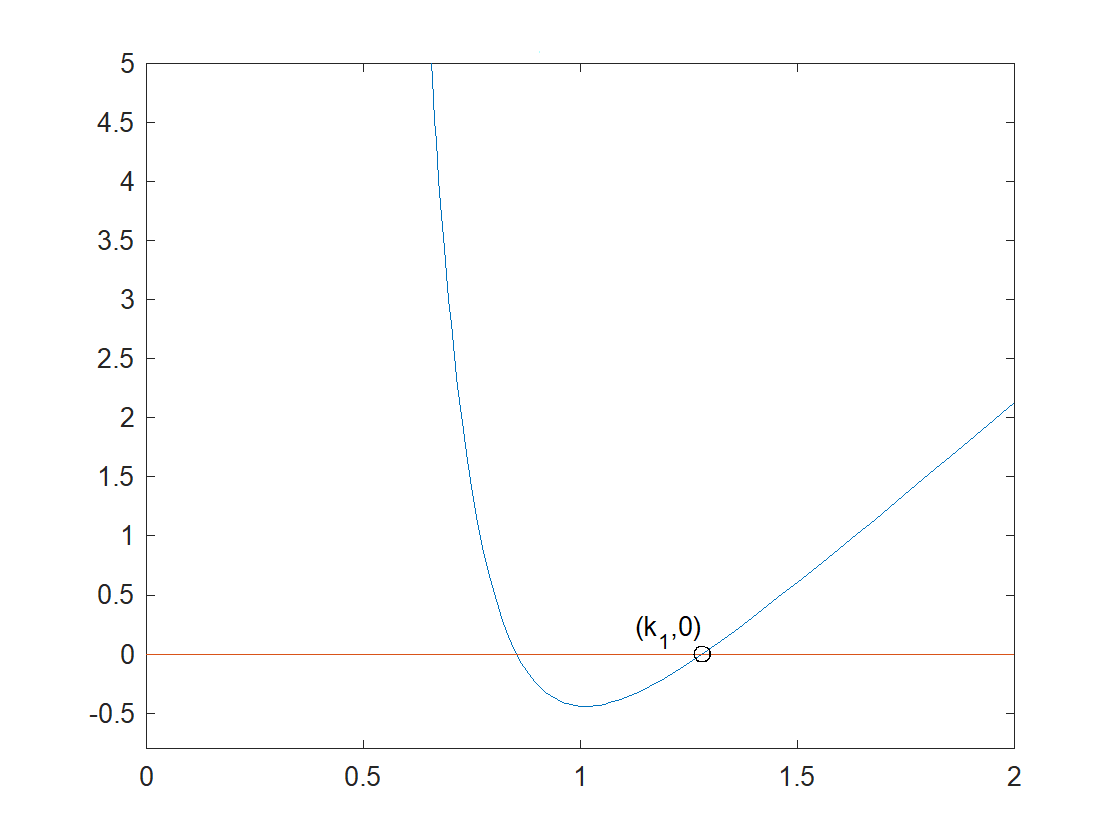}
\caption{\scriptsize Example of solution to $ f(k_1)=0$.}
\end{figure}
\par Observe that $\phi$, $\phi'$, and $\phi''$ are continuous on $[k_1, k_2]$.  Since $\phi(k_2)=\phi'(k_2)=0$ and $\phi''(k_2)>0$, there exists $\varepsilon_1>0$ such that $\phi$ is nonnegative, decreasing, and convex over the interval $(k_2-\varepsilon_1,k_2)$.  Since $\phi(k_1)\ge 0$ and $\phi'(k_1)\ge0$, there exists $\varepsilon_2>0$ such that $\phi$ is nonnegative and increasing on $(k_1, k_1+\varepsilon_2)$; moreover, $k_1+\varepsilon_2<k_2-\varepsilon_1$.  Suppose, if possible, there exists $y\in(k_1+\varepsilon_2, k_2-\varepsilon_1)$ such that $\phi(y)<0$.
\par Note that $\phi\left(k_1+\frac{\varepsilon_2}{2}\right)>0$.  Then by Intermediate Value Theorem, there exists $y_1\in\left(k_1+\frac{\varepsilon_2}{2},y\right)$ such that $\phi(y_1)=0$.  Similarly, since $\phi\left(k_2-\frac{\varepsilon_1}{2}\right)>0$, there exists $y_2\in\left(y,k_2-\frac{\varepsilon_1}{2}\right)$ such that $\phi(y_2)=0$.
Note also that $\phi'\left(k_1+\frac{\varepsilon_2}{2}\right)>0$ and $\phi'(y_1)<0$.  So, by Intermediate Value Theorem, there exists $\widetilde{y_1} \in \left(k_1+\frac{\varepsilon_2}{2}, y_1\right)$ such that $\phi'(\widetilde{y_1})=0$.  Similarly, since $\phi'(y_2)>0$, there exists $\widetilde{y_2}\in(y_1,y_2)$ such that $\phi'(\widetilde{y_2})=0$.  Also, since $\phi'\left(k_2-\frac{\varepsilon_1}{2}\right)<0$, there exists $\widetilde{y_3}\in\left(y_2, k_2-\frac{\varepsilon_1}{2}\right)$ such that $\phi'(\widetilde{y_3})=0$.
\par Finally, since $\phi'(\widetilde{y_1})=\phi'(\widetilde{y_2})=0$, by Rolle's Theorem, there exists $y_1^*\in(\widetilde{y_1},\widetilde{y_2})$ such that $\phi''(y_1^*)=0$.  Similarly, since $\phi'(\widetilde{y_3})=0$, there exists $y_2^*\in(\widetilde{y_2},\widetilde{y_3})$ such that $\phi''(y_2^*)=0$.  But this is a contradiction, because $y_1^*\in[k_1, k_2]$, $y_2^*\in[k_1,k_2]$, but $y_1^*\ne y_2^*$;  whereas the equation $\phi''(y)=0$ has exactly one solution in the interval $[k_1, k_2]$.
\par Hence, $\phi(y)=C_1y^{\delta_1}-C_2y^{\delta_2}+\Bb-\Bs y\ge 0$ on $\Gamma_2$.  That is, $w_0(y)-w_1(y)+\Bb-\Bs y\ge 0$ for all $y\in\Gamma_2$.

\par The solutions of the HJB equations have the form:
\begin{align}
w_0(y)&=
\begin{cases}
\Bs-\Bb y, &\indent\indent\indent\indent\indent\indent\indent ~\text{if}~ 0<y<k_2\\
\displaystyle{\left(\frac{\Bs}{1-\delta_2}\right)^{1-\delta_2}\left(\frac{\Bb}{-\delta_2}\right)^{\delta_2}y^{\delta_2}}, &\indent\indent\indent\indent\indent\indent\indent~\text{if}~ y\ge k_2\\
\end{cases}\\
w_1(y)&=
\begin{cases}
\displaystyle{\left[1-\left(\frac{\Bs}{\Bb}\right)^{-\delta_2}\left(\frac{-\delta_2}{1-\delta_2}\right)^{1-\delta_2}k_1^{\delta_2-1}\right]\left(\frac{\Bs}{\delta_1}\right)k_1^{1-\delta_1}y^{\delta_1}}, & \/\ ~\text{if}~ 0<y<k_1\\
\displaystyle{\left(\frac{\Bs}{1-\delta_2}\right)^{1-\delta_2}\left(\frac{\Bb}{-\delta_2}\right)^{\delta_2}y^{\delta_2}-\Bb+\Bs y}, &\/\ ~\text{if}~ y\ge k_1\\
\end{cases}
\end{align}
\section{Verification Theorem}
\begin{thmm}
We have $v_i(x_1,x_2)=x_1w_i\left(\dfrac{x_1}{x_2}\right)=V_i(x_1,x_2)$, $i=0,1$.  Moreover, if initially $i=0$, let $\Lambda_0^*=(\tau_1^*, \tau_2^*)$ such that
\begin{center}
$\tau_1^*=\inf\{t\ge0~|~(X_t^1, X_t^2)\in\Gamma_3\}$, $\tau_2^*=\inf\{t\ge\tau_1^*~|~(X_t^1, X_t^2)\in\Gamma_1\}$. 
\end{center}
Similarly, if initially $i=1$, let $\Lambda_1^*=(\tau_0^*)$ such that
\begin{center}$\tau_0^*=\inf\{t\ge0~|~(X_t^1, X_t^2)\in\Gamma_1\}$. 
\end{center}
Then $\Lambda_0^*$ and $\Lambda_1^*$ are optimal.
\end{thmm}
\begin{proof}The proof is divided into 4 steps.\\
\textbf{Step 1:~}  $C_1>0$, $C_2>0$, $v_0(x_1,x_2)\ge0$.
\par Clearly $C_2=\left(\dfrac{\beta_s}{1-\delta_2}\right)^{1-\delta_2}\left(\dfrac{\beta_b}{-\delta_2}\right)^{\delta_2}>0$.  Also, $C_1=\dfrac{C_2\delta_2k_2^{\delta_2-1}+\beta_s}{\delta_1k_2^{\delta_1-1}}>0$ has previously been established.  Now,

$v_0(x_1,x_2)=x_1w_0\left(\dfrac{x_2}{x_1}\right)=$ 
$\begin{cases}
        C_1x_2^{\delta_1}x_1^{1-\delta_1}, & \text{on } \Gamma_1\cup\Gamma_2\\
        C_2x_2^{\delta_2}x_1^{1-\delta_2}-\beta_bx_1+\beta_sx_2, & \text{on } \Gamma_3
\end{cases}$\\
\\
Hence to show $v_0(x_1,x_2)\ge0$, it suffices to show $w_0(y)\ge0$ on $\Gamma_3$.  The continuity of $w_0$ and $w_0^{\prime}$ yield
$w_0(k_2)=C_2k_2^{\delta_2}-\beta_b+\beta_sk_2=C_1k_2^{\delta_1}>0$ and $w_0^{\prime}(k_2)=C_2\delta_2k_2^{\delta_2-1}+\beta_s=C_1\delta_1k_2^{\delta_1-1}>0$.  Also, $w_0^{\prime\prime}(y)=C_2\delta_2(\delta_2-1)y^{\delta_2-2}>0$ for all $y>0$.  In particular, since $w_0^{\prime\prime}(y)>0$ for all $y\in\Gamma_3$, we know $w_0^{\prime}(y)$ is increasing on $\Gamma_3$.  And since $w_0^{\prime}(k_2)>0$, it must be that $w_0^{\prime}(y)>0$ for all $y\in\Gamma_3$.  This in turn implies that $w_0(y)$ is increasing on $\Gamma_3$.  Since we know $w_0(k_2)>0$, it must be that $w_0(y)>0$ for all $y\in\Gamma_3$.\\
\\
\textbf{Step 2:~}  $-Ax_1-Bx_2\le v_i(x_1,x_2)\le Ax_1+Bx_2$\textbf{,} $i=0,1$.
\par Let $i=0$.  On $\Gamma_1\cup\Gamma_2$, we have $0\le v_0(x_1,x_2)=C_1x_1^{1-\delta_1}x_2^{\delta_1}\le C_1x_1k_2^{\delta_1}$.  On $\Gamma_3$, $-\beta_bx_1+\beta_sx_2\le v_0(x_1,x_2)=C_2x_1^{1-\delta_2}x_2^{\delta_2}-\beta_bx_1+\beta_sx_2\le C_2x_1k_1^{\delta_2}-\beta_bx_1+\beta_sx_2$.  Hence we can choose suitable $A$ and $B$ so the inequalities hold when $i=0$. 
\par Let $i=1$.  On $\Gamma_2\cup\Gamma_3$, we have $0\le v_1(x_1,x_2)=C_2x_1^{1-\delta_2}x_2^{\delta_2}\le C_2x_1k_1^{\delta_2}$.  On $\Gamma_1$, $-\beta_bx_2\le v_1(x_1,x_2)=\beta_sx_1-\beta_bx_2\le\beta_sx_1$.  So again we can choose suitable $A$ and $B$ so the inequalities hold when $i=1$.\\
\\
\textbf{Step 3:~}  $v_i(x_1,x_2)\ge J_i(x_1,x_2,\Lambda_i)$.
\par The functions $v_0$ and $v_1$ are continuously differentiable on the entire region $\{x_1>0,~x_2>0\}$ and twice continuously differentiable on the interior of $\Gamma_i$, $i=1,2,3$.  In addition, they satisfy
\begin{center}
    $0 \le (\rho-\mathcal{L})w_0(y)$ \\
    $0 \le (\rho-\mathcal{L})w_1(y)$ \\
    $-\beta_b+\beta_sy \le w_0(y)-w_1(y)\le w_0(y)-\beta_s+\beta_by$ \\
\end{center}
\par In particular, $\rho v_i(x)-\mathcal{A}v_i(x)\ge0$, $i=0,1$, whenever they are twice continuously differentiable.  Using these inequalities, Dynkin's formula, and Fatou's Lemma, as in \O ksendal, we have $\mathbb{E}\left[e^{-\rho(\theta_1\land N)}v_i(X_{\theta_1\land N}^1,X_{\theta_1\land N}^2) \right]\ge\mathbb{E}\left[e^{-\rho(\theta_2\land N)}v_i(X_{\theta_2\land N}^1,X_{\theta_2\land N}^2) \right]$ for any stopping times $0\le\theta_1\le\theta_2$, almost surely, and any $N$.\\
\indent For each $j=1,2$, 
\begin{align*}
& \mathbb{E}\left[e^{-\rho(\theta_j\land N)}v_i(X_{\theta_j\land N}^1,X_{\theta_j\land N}^2) \right]  \\
& =\mathbb{E}\left[e^{-\rho(\theta_j\land N)}v_i(X_{\theta_j\land N}^1,X_{\theta_j\land N}^2)\mathbb{I}_{\{\theta_j<\infty\}} \right]+\mathbb{E}\left[e^{-\rho(\theta_j\land N)}v_i(X_{\theta_j\land N}^1,X_{\theta_j\land N}^2)\mathbb{I}_{\{\theta_j=\infty\}} \right]\\
& = \mathbb{E}\left[e^{-\rho(\theta_j\land N)}v_i(X_{\theta_j\land N}^1,X_{\theta_j\land N}^2)\mathbb{I}_{\{\theta_j<\infty\}} \right]+\mathbb{E}\left[e^{-\rho N}v_i(X_{ N}^1,X_{N}^2)\mathbb{I}_{\{\theta_j=\infty\}} \right]
\end{align*}
\par In view of Step 2, the second term on the right hand side converges to zero because both $\mathbb{E}\left[e^{-\rho N}X_N^1\right]$ and $\mathbb{E}\left[e^{-\rho N}X_N^2\right]$ go to zero as $N\to\infty$.  Also, $e^{-\rho (\theta_j\land N)}v_i(X_{\theta_j \land N}^1, X_{\theta_j \land N}^2)\mathbb{I}_{\{\theta_j<\infty\}}\to e^{-\rho\theta_j}v_i(X_{\theta_j}^1, X_{\theta_j}^2)\mathbb{I}_{\{\theta_j<\infty\}}$ almost surely as $N\to \infty$.\\
\indent By showing the existence of $\gamma_i$, $i=1,2$ such that 
\begin{align*}
    & \sup_n\mathbb{E}\left[\left(e^{-\rho(\theta_j\land N)}X_{\theta_j\land N}^1\right)^{\gamma_1}\right]<\infty,\\
    & \sup_n\mathbb{E}\left[\left(e^{-\rho(\theta_j\land N)}X_{\theta_j\land N}^2\right)^{\gamma_2}\right]<\infty,
\end{align*}
we can show that both $\left\{e^{-\rho(\theta_j\land N)}X_{\theta_j\land N}^1\right\}$ and $\left\{e^{-\rho(\theta_j\land N)}X_{\theta_j\land N}^2\right\}$ are uniformly integrable.  Hence we obtain the uniform integrability of $\left\{e^{-\rho(\theta_j\land N)}v_i(X_{\theta_j\land N}^1,X_{\theta_j\land N}^2)\right\}$ and send $N$ to $\infty$ to obtain $$\mathbb{E}\left[e^{-\rho\theta_1}v_i(X_{\theta_1}^1, X_{\theta_1}^2)\mathbb{I}_{\{\theta_1<\infty\}}\right]\ge\mathbb{E}\left[e^{-\rho\theta_2}v_i(X_{\theta_2}^1, X_{\theta_2}^2)\mathbb{I}_{\{\theta_2<\infty\}}\right],$$ for $i=0,1$.
\par Given $\Lambda_0=(\tau_1,\tau_2)$, $\Lambda_1=(\tau_0)$
\begin{align*}
    v_0(x_1,x_2)&\ge\mathbb{E}\left[e^{-\rho\tau_1}v_0(X_{\tau_1}^1, X_{\tau_1}^2)\mathbb{I}_{\{\tau_1<\infty\}}\right]\\
    & \ge\mathbb{E}\left[e^{-\rho\tau_1}\left(v_1(X_{\tau_1}^1, X_{\tau_1}^2)-\beta_b X_{\tau_1}^1+\beta_s X_{\tau_1}^2\right)\mathbb{I}_{\{\tau_1<\infty\}}\right]\\
    & =\mathbb{E}\left[e^{-\rho\tau_1} v_1(X_{\tau_1}^1, X_{\tau_1}^2)\mathbb{I}_{\{\tau_1<\infty\}}-e^{-\rho\tau_1}\left(\beta_b X_{\tau_1}^1+\beta_s X_{\tau_1}^2\right) \mathbb{I}_{\{\tau_1<\infty\}}\right]\\
    & \ge\mathbb{E}\left[e^{-\rho\tau_2} v_1(X_{\tau_2}^1, X_{\tau_2}^2)\mathbb{I}_{\{\tau_2<\infty\}}-e^{-\rho\tau_1}\left(\beta_b X_{\tau_1}^1+\beta_s X_{\tau_1}^2\right) \mathbb{I}_{\{\tau_1<\infty\}}\right]\\
    & \ge\mathbb{E}\left[e^{-\rho\tau_2}\left(\beta_sX_{\tau_2}^1-\beta_bX_{\tau_2}^2\right)\mathbb{I}_{\{\tau_2<\infty\}}-e^{-\rho\tau_1}\left(\beta_b X_{\tau_1}^1+\beta_s X_{\tau_1}^2\right) \mathbb{I}_{\{\tau_1<\infty\}}\right]\\
    & = J_0(x_1,x_2,\Lambda_0)\\
    v_1(x_1,x_2)&\ge\mathbb{E}\left[e^{-\rho\tau_1}v_1(X_{\tau_0}^1, X_{\tau_0}^2)\mathbb{I}_{\{\tau_0<\infty\}}\right]\\
    & \ge\mathbb{E}\left[e^{-\rho\tau_0}\left(\beta_s X_{\tau_0}^1-\beta_b X_{\tau_0}^2\right)\mathbb{I}_{\{\tau_0<\infty\}}\right]\\
    & = J_1(x_1,x_2,\Lambda_1)
\end{align*}
\textbf{Step 4:~}  $v_i(x_1,x_2)= J_i(x_1,x_2,\Lambda_i^*)$.
\par Let $i=0$.  Define $\tau_1^*=\inf{\{t\ge 0~|~(X_t^1, X_t^2)\in\Gamma_3\}}$, $\tau_2^*=\inf{\{t\ge \tau_1^*~|~(X_t^1, X_t^2)\in\Gamma_1\}}$.  We apply Dynkin's formula and notice that, for each $n$, $v_0(x_1,x_2)=\mathbb{E}\left[e^{-\rho(\tau_1^*\land n)}v_0(X_{\tau_1^*\land n}^1, X_{\tau_1^*\land n}^2)\right]$.  Note also that $\underset{n\to\infty}{\lim}\mathbb{E}\left[e^{-\rho(\tau_1^*\land n)}v_0(X_{\tau_1^*\land n}^1, X_{\tau_1^*\land n}^2)\right]=\mathbb{E}\left[e^{-\rho \tau_1^*}v_0(X_{\tau_1^*}^1, X_{\tau_1^*}^2)\mathbb{I}_{\{\tau_1^*<\infty\}}\right]$. It follows that
\begin{align*}
v_0(x_1,x_2)&=\mathbb{E}\left[e^{-\rho \tau_1^*}v_0(X_{\tau_1^*}^1, X_{\tau_1^*}^2)\mathbb{I}_{\{\tau_1^*<\infty\}}\right]\\
&=\mathbb{E}\left[e^{-\rho \tau_1^*}\left(v_1(X_{\tau_1^*}^1, X_{\tau_1^*}^2)-\beta_bX_{\tau_1^*}^1+\beta_sX_{\tau_1^*}^2\right)\mathbb{I}_{\{\tau_1^*<\infty\}}\right].
\end{align*}
We have also 
\begin{align*}
\mathbb{E}\left[e^{-\rho \tau_1^*}v_1(X_{\tau_1^*}^1, X_{\tau_1^*}^2)\mathbb{I}_{\{\tau_1^*<\infty\}}\right]&=\mathbb{E}\left[e^{-\rho \tau_2^*}v_1(X_{\tau_2^*}^1, X_{\tau_2^*}^2)\mathbb{I}_{\{\tau_2^*<\infty\}}\right]\\
&=\mathbb{E}\left[e^{-\rho \tau_2^*}\left(\beta_sX_{\tau_2^*}^1-\beta_bX_{\tau_2^*}^2\right)\mathbb{I}_{\{\tau_2^*<\infty\}}\right]
\end{align*}
Combine these to obtain
\begin{align*}
v_0(x_1,x_2)&=\mathbb{E}\left[e^{-\rho \tau_2^*}\left(\beta_sX_{\tau_2^*}^1-\beta_bX_{\tau_2^*}^2\right)\mathbb{I}_{\{\tau_2^*<\infty\}}-\left(\beta_bX_{\tau_1^*}^1+\beta_sX_{\tau_1^*}^2\right)\mathbb{I}_{\{\tau_1^*<\infty\}}\right]  \\
&= J_0(x_1,x_2,\Lambda_0^*).
\end{align*}

\par Let $i=1$.  Define $\tau_0^*=\inf{\{t\ge 0~|~(X_t^1, X_t^2)\in\Gamma_1\}}$.  We apply Dynkin's formula and notice that, for each $n$, $v_1(x_1,x_2)=\mathbb{E}\left[e^{-\rho(\tau_0^*\land n)}v_1(X_{\tau_0^*\land n}^1, X_{\tau_0^*\land n}^2)\right]$.  \\\
Note also that $\underset{n\to\infty}{\lim}\mathbb{E}\left[e^{-\rho(\tau_0^*\land n)}v_1(X_{\tau_0^*\land n}^1, X_{\tau_0^*\land n}^2)\right]=\mathbb{E}\left[e^{-\rho \tau_0^*}v_1(X_{\tau_0^*}^1, X_{\tau_0^*}^2)\mathbb{I}_{\{\tau_0^*<\infty\}}\right]$. It follows that
\begin{align*}
v_1(x_1,x_2)&=\mathbb{E}\left[e^{-\rho \tau_0^*}v_1(X_{\tau_0^*}^1, X_{\tau_0^*}^2)\mathbb{I}_{\{\tau_0^*<\infty\}}\right]\\
&=\mathbb{E}\left[e^{-\rho \tau_0^*}\left(\beta_sX_{\tau_0^*}^1-\beta_bX_{\tau_0^*}^2\right)\mathbb{I}_{\{\tau_0^*<\infty\}}\right]\\
&= J_1(x_1,x_2,\Lambda_1^*).
\end{align*}  
\end{proof}
\section{A Numerical Example}
\begin{figure}
\includegraphics[width=0.75\textwidth]{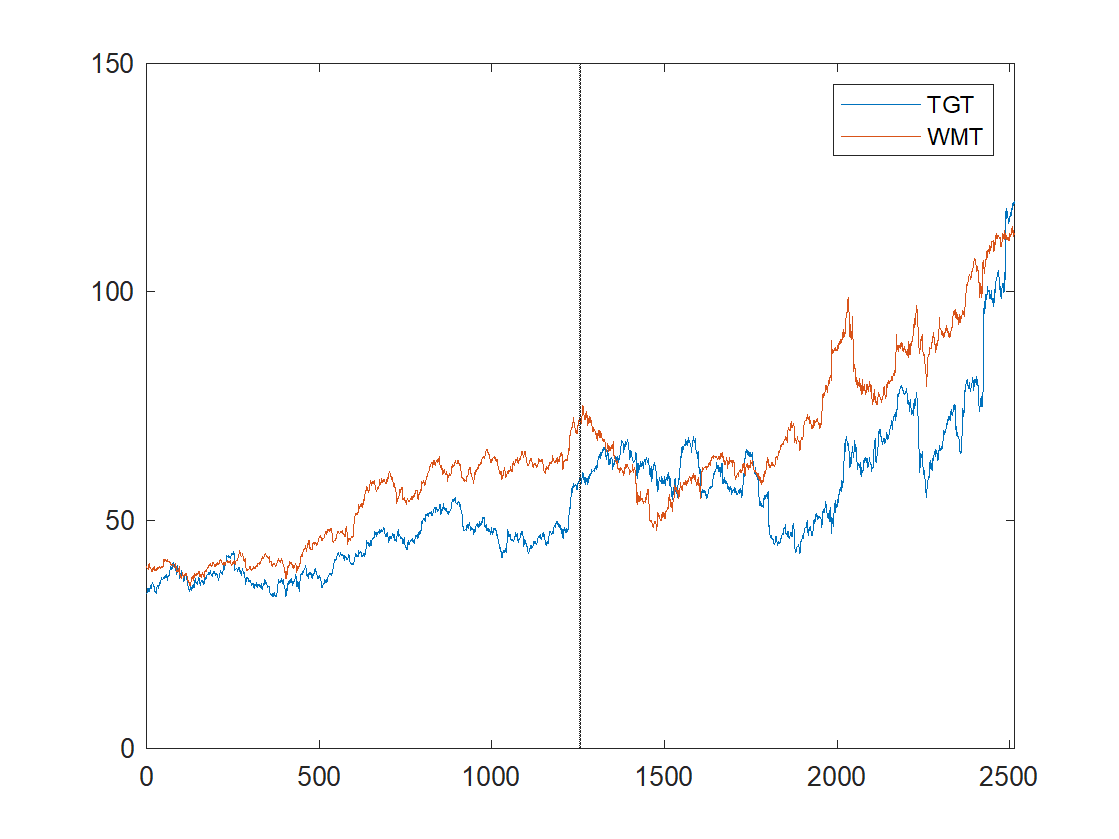}
\caption{\scriptsize Closing Prices of TGT and WMT}
\end{figure}

\par We consider adjusted closing price data for Walmart (WMT) and Target (TGT) from 2010 to 2020.  The first half of the data is used to calibrate the model, and the second half is used to test the results.  Using a least-squares method, we obtain the following parameters: $\mu_1=0.09696$, $\mu_2=0.14347$, $\sigma_{11}=0.19082$, $\sigma_{12}=0.04036$, $\sigma_{21}=0.04036$, and $\sigma_{22}=0.13988$.  We specify $K=0.001$ and $\rho=0.5$.  Then we find $k_1=0.85527$, and $k_2=1.28061$.
\par
Next we examine the dependence of $k_1$ and $k_2$ on the parameters by varying each.  We see that $k_1$ and $k_2$ both decrease in $\mu_1$.  This leads to a larger buying region, $\Gamma_3$.  On the other hand, both $k_1$ and $k_2$ increase in $\mu_2$.  This creates a larger $\Gamma_1$ and, hence, encourages early exit.  When varying $\sigma_{11}$ and $\sigma_{22}$, we find that $k_2$ increases while $k_1$ decreases, in both $\sigma_{11}$ and $\sigma_{22}$.  This leads to a smaller buying zone, $\Gamma_1$, due to the increased risk, as well as a smaller selling zone, $\Gamma_3$, because there is more price movement overall.  However, as $\sigma_{12}=\sigma_{21}$ increases, we find that $k_2$ decreases, while $k_1$ increases.  The greater correlation leads to a larger $\Gamma_1$, and hence more opportunity for buying, as well as a larger $\Gamma_3$, and hence more opportunity for selling.  Since $r$ represents the rate at which money loses value over time, $k_2$ decreases in $r$, while $k_1$ increases in $r$, reflecting the fact that we are less likely to want to hold in this case.  Finally, larger transaction costs discourage trading.  Naturally, as $K$ increases, $k_2$ increases and $k_1$ decreases.  

\begin{table}[H]
\begin{center}
\begin{tabular}{ |c|c|c|c|c|c| } 
\hline
$\mu_1$ & $-0.00304$ & $0.04696$ & $0.09696$ & $0.14696$ & $0.19696$ \\
 \hline
$k_1$ & $0.91380$ & $0.89057$ & $0.85527$ & $0.80194$ & $0.72644$  \\ 
\hline
$k_2$ & $1.54188$ & $1.41541$ & $1.28061$ & $1.12891$ & $0.96334$ \\
\hline
\end{tabular}
\caption{$k_1$ and $k_2$ with varying $\mu_1$}
\end{center}
\end{table}

\begin{table}[H]
\begin{center}
\begin{tabular}{ |c|c|c|c|c|c| } 
\hline
$\mu_2$ & $0.04347$ & $0.09347$ & $0.14347$ & $0.19347$ & $0.24347$ \\
\hline
$k_1$ & $0.76457$ & $0.81341$ & $0.85527$ & $0.88736$ & $0.91037$  \\ 
\hline
$k_2$ & $0.98771$ & $1.12128$ & $1.28061$ & $1.47155$ & $1.72474$ \\ 
\hline
\end{tabular}
\caption{$k_1$ and $k_2$ with varying $\mu_2$}
\end{center}
\end{table}

\begin{table}[H]
\begin{center}
\begin{tabular}{ |c|c|c|c|c|c| } 
\hline
$\sigma_{11}$ & $0.09082$ & $0.14082$ & $0.19082$ & $0.24082$ & $0.29082$\\
\hline
$k_1$ & $0.92069$ & $0.89220$ & $0.85527$ & $0.81532$ & $0.77497$\\
\hline
$k_2$ & $1.21691$ & $1.24468$ & $1.28061$ & $1.32066$ & $1.36327$\\
\hline
\end{tabular}
\caption{$k_1$ and $k_2$ with varying $\sigma_{11}$}
\end{center}
\end{table}

\begin{table}[H]
\begin{center}
\begin{tabular}{ |c|c|c|c|c|c| } 
\hline
$\sigma_{22}$ & $0.03988$ & $	0.08988$ & $0.13988$ & $0.18988$ & $0.23988$\\
\hline
$k_1$ & $0.88356$ & $0.87601$ & $0.85527$ & $0.82593$ & $0.79206$\\
\hline
$k_2$ & $1.25304$ & $1.26036$ & $1.28061$ & $1.30985$ & $1.34491$\\
\hline
\end{tabular}
\caption{$k_1$ and $k_2$ with varying $\sigma_{22}$}
\end{center}
\end{table}

\begin{table}[H]
\begin{center}
\begin{tabular}{ |c|c|c|c|c|c| } 
\hline
$\sigma_{12}$ & $-0.05964$ & $-0.00964$ & $0.04036$ & $0.09036$ & $0.14036$\\
\hline
$k_1$ & $0.73242$ & $0.79189$ & $0.85527$ & $0.92029$ & $0.97527$\\
\hline
$k_2$ & $1.41132$ & $1.34509$ & $1.28061$ & $1.21730$ & $1.15901$\\
\hline
\end{tabular}
\caption{$k_1$ and $k_2$ with varying $\sigma_{12}=\sigma_{21}$}
\end{center}
\end{table}

\begin{table}[H]
\begin{center}
\begin{tabular}{ |c|c|c|c|c|c| } 
\hline
$r$ & $0.4$ & $0.45$ & $0.5$ & $0.55$ & $0.6$\\
\hline
$k_1$ & $0.84068$ & $0.84858$ & $0.85527$ & $0.86105$ & $0.86611$\\
\hline
$k_2$ & $1.36281$ & $1.31541$ & $1.28061$ & $1.25387$ & $1.23262$\\
\hline
\end{tabular}
\caption{$k_1$ and $k_2$ with varying $r$}
\end{center}
\end{table}

\begin{table}[H]
\begin{center}
\begin{tabular}{ |c|c|c|c|c|c| } 
\hline
$K$ & $0.0000$ & $0.0005$ & $0.0010$ & $0.0015$ & $0.0020$\\
\hline
$k_1$ & $0.85698$ & $0.85613$ & $0.85527$ & $0.85442$ & $0.85356$\\
\hline
$k_2$ & $1.27670$ & $1.27866$ & $1.28061$ & $1.28254$ & $1.28447$\\
\hline
\end{tabular}
\caption{$k_1$ and $k_2$ with varying $K$}
\end{center}
\end{table}

\section{A Second Approach to Formulating the Problem}

\par Having previously allowed the initial pairs position to be long or flat, a natural next question to consider is the short side of pairs trading.  So, we begin again with the same stochastic differential equation as in (\ref{SDE}) and the same partial differential operator as in (\ref{pdo}), but now we allow our intial pairs position to be flat ($i=0$), long ($i=1$), or short ($i=-1$). If initially we are short in $\textbf{Z}$, we will buy one share of $\textbf{Z}$, i.e. buy one share of $\textbf{S}^1$ and sell one share of $\textbf{S}^2$, at some time $\tau_0$, which will conclude our trading activity.  If initially we are long in $\textbf{Z}$, we will sell one share of $\textbf{Z}$, i.e. sell $\textbf{S}^1$ and buy $\textbf{S}^2$ at some time $\tau_0$, which will conclude our trading activity.  Otherwise, if initially we are flat, we can either go long or short one share in $\textbf{Z}$ at some time $\tau_1$.  Depending on our activity at time $\tau_1,$ we would then either sell $\textbf{S}^1$ and buy $\textbf{S}^2$ (if long) or buy $\textbf{S}^1$ and sell $\textbf{S}^2$ (if short) at some time $\tau_2\ge\tau_1,$ thus concluding our trading activity. Hence, for $x_1,x_2>0$, the HJB equations become

\par We seek thresholds $k_1,$ $k_2,$ $k_3,$ and $k_4$ for buying and selling $\textbf{Z}.$  Let $k_1$ indicate the price at which we will sell one share of $\textbf{Z}$ when the net position is flat.  Let $k_2$ indicate the price at which we will sell one share of $\textbf{Z}$ when the net position is long.  Let $k_3$ indicate the price at which we will buy one share of $\textbf{Z}$ when the net position is short.  Let $k_4$ indicate the price at which we will buy one share of $\textbf{Z}$ when the net position is flat.  Then define the following function:
\beq{u}
u(x_1,x_2,i)=\begin{cases} 
-1 & i=0~~\text{and}~~x_2 \leq x_1 k_1\\
-1 & i=1~~\text{and}~~x_2 \leq x_1 k_2\\
1 & i=-1~~\text{and}~~x_2 \ge x_1 k_3\\
1 & i=0~~\text{and}~~x_2 \ge x_1 k_4\\
\end{cases}
\eeq
\par Let $K$ denote the fixed percentage of transaction costs associate with buying or selling of stocks.  Then given the initial state $(x_1,x_2),$ the initial net position $i=-1,0,1,$ and the decision sequences $\Lambda_{-1}=(\tau_0),$ $\Lambda_1=(\tau_0)$ and $\Lambda_0=(\tau_1,\tau_2)$, the resulting reward functions are
\begin{align}
J_{-1}(x_1,x_2,\tau_0) =&\mathbb{E}\left[-e^{-\rho\tau_0}\left(\beta_b X_{\tau_0}^1-\beta_s X_{\tau_0}^2\right)\mathbb{I}_{\{\tau_0<\infty\}}\right]\\
J_0(x_1,x_2,\tau_1,\tau_2,u) =&\mathbb{E}\big[\big\{e^{-\rho\tau_2}\left(\beta_s X_{\tau_2}^1-\beta_b X_{\tau_2}^2\right)\mathbb{I}_{\{\tau_2<\infty\}}
-e^{-\rho\tau_1}\left(\beta_b X_{\tau_1}^1-\beta_s X_{\tau_1}^2\right)\mathbb{I}_{\{\tau_1<\infty\}}\big\}\mathbb{I}_{\{u=1\}}\\
&+\big\{e^{-\rho\tau_1}\left(\beta_s X_{\tau_1}^1-\beta_b X_{\tau_1}^2\right)\mathbb{I}_{\{\tau_1<\infty\}}
-e^{-\rho\tau_2}\left(\beta_b X_{\tau_2}^1-\beta_s X_{\tau_2}^2\right)\mathbb{I}_{\{\tau_2<\infty\}}\big\}\mathbb{I}_{\{u=-1\}}\big]\notag\\
J_1(x_1,x_2,\tau_0) =&\mathbb{E}\left[e^{-\rho\tau_0}\left(\beta_s X_{\tau_0}^1-\beta_b X_{\tau_0}^2\right)\mathbb{I}_{\{\tau_0<\infty\}}\right]
\end{align}
\par For $i=-1,0,1,$ let $V_i(x_1,x_2)$ denote the value functions with initial state $(X_0^1, X_0^2)=(x_1,x_2)$ and initial net positions $i=-1,0,1.$  That is, $V_i(x_1,x_2)=\underset{\Lambda_i}{\sup}{~J_i(x_1,x_2,\Lambda_i)}$.

\section{Properties of the Value Functions}
In this section, we establish basic properties of the value functions.  
\begin{lemm}
For all $x_1$, $x_2>0$, we have
$$\beta_sx_1-\beta_bx_2\le V_1(x_1,x_2)\le x_1,$$
$$\beta_sx_2-\beta_bx_1\le V_{-1}(x_1,x_2)\le x_2,\text{~and}$$
$$0\le V_0(x_1,x_2)\le 4x_1+4x_2.$$
\end{lemm}
\begin{proof}
Note that for all $x_1, x_2>0,$ $V_1(x_1,x_2)\ge J_1(x_1,x_2,\tau_0)=\mathbb{E}\left[e^{-\rho\tau_0}\left(\beta_sX_{\tau_0}^1-\beta_bX_{\tau_0}^2\right)\mathbb{I}_{\{\tau_0<\infty\}}\right]$.  In particular, $$V_1(x_1,x_2)\ge J_1(x_1,x_2,0)=\beta_sx_1-\beta_bx_2.$$  Similarly, $V_{-1}(x_1,x_2)\ge J_{-1}(x_1,x_2,\tau_0)=\mathbb{E}\left[-e^{-\rho\tau_0}\left(\beta_bX_{\tau_0}^1-\beta_sX_{\tau_0}^2\right)\mathbb{I}_{\{\tau_0<\infty\}}\right]$.  In particular, $$V_{-1}(x_1,x_2)\ge J_{-1}(x_1,x_2,0)=\beta_sx_2-\beta_bx_1.$$  Finally, 
\begin{align*}
V_0(x_1,x_2)&\ge J_0(x_1,x_2,\tau_1, \tau_2,u)\\
&=\mathbb{E}\big[\big\{e^{-\rho\tau_2}\left(\beta_sX_{\tau_2}^1-\beta_bX_{\tau_2}^2\right)\mathbb{I}_{\{\tau_2<\infty\}}- e^{-\rho\tau_1}\left(\beta_bX_{\tau_1}^1-\beta_sX_{\tau_1}^2\right)\mathbb{I}_{\{\tau_1<\infty\}}\big\}
\mathbb{I}_{\{u=1\}}\\
&\indent+\big\{e^{-\rho\tau_1}\left(\beta_sX_{\tau_1}^1-\beta_bX_{\tau_1}^2\right)\mathbb{I}_{\{\tau_1\infty\}}- e^{-\rho\tau_2}\left(\beta_bX_{\tau_2}^1-\beta_sX_{\tau_2}^2\right)\mathbb{I}_{\{\tau_2<\infty\}}\big\}\mathbb{I}_{\{u=-1\}}\big].
\end{align*} 
 Clearly, $V_0(x_2,x_2)\ge 0$ by definition and taking $\tau_1=\infty.$  Now, for all $\tau_0>0,$
$  J_1(x_1,x_2,\tau_0) $
\begin{align*}
& =\mathbb{E}\left[e^{-\rho\tau_0}\left(\beta_sX_{\tau_0}^1-\beta_bX_{\tau_0}^2\right)\mathbb{I}_{\{\tau_0<\infty\}}\right]\\
& \le\mathbb{E}\left[e^{-\rho\tau_0}\left(X_{\tau_0}^1-X_{\tau_0}^2\right)\mathbb{I}_{\{\tau_0<\infty\}}\right]\\
& =x_1+\mathbb{E}\left[\int_0^{\tau_0}\left(-\rho+\mu_1\right)e^{-\rho t}X_t^1 \dt \mathbb{I}_{\{\tau_0<\infty\}}\right]-x_2-\mathbb{E}\left[\int_0^{\tau_0}\left(-\rho+\mu_2\right)e^{-\rho t}X_t^2 \dt \mathbb{I}_{\{\tau_0<\infty\}}\right]\\
& \le x_1-x_2-\mathbb{E}\left[\int_0^{\tau_0}\left(-\rho+\mu_2\right)e^{-\rho t}X_t^2 \dt \mathbb{I}_{\{\tau_0<\infty\}}\right]\\
& \le x_1-x_2+\mathbb{E}\left[\int_0^{\infty}\left(\rho-\mu_2\right)e^{-\rho t}X_t^2 \dt \right]\\
& = x_1.
\end{align*}
Also, for all $\tau_0>0,$
$J_{-1}(x_1,x_2,\tau_0) $
\begin{align*}
& =\mathbb{E}\left[-e^{-\rho\tau_0}\left(\beta_bX_{\tau_0}^1-\beta_sX_{\tau_0}^2\right)\mathbb{I}_{\{\tau_0<\infty\}}\right]\\
& \le\mathbb{E}\left[-e^{-\rho\tau_0}\left(X_{\tau_0}^1-X_{\tau_0}^2\right)\mathbb{I}_{\{\tau_0<\infty\}}\right]\\
& =x_2+\mathbb{E}\left[\int_0^{\tau_0}\left(-\rho+\mu_2\right)e^{-\rho t}X_t^2 \dt \mathbb{I}_{\{\tau_0<\infty\}}\right]-x_1-\mathbb{E}\left[\int_0^{\tau_0}\left(-\rho+\mu_1\right)e^{-\rho t}X_t^1 \dt \mathbb{I}_{\{\tau_0<\infty\}}\right]\\
& \le x_2-x_1-\mathbb{E}\left[\int_0^{\tau_0}\left(-\rho+\mu_1\right)e^{-\rho t}X_t^1 \dt \mathbb{I}_{\{\tau_0<\infty\}}\right]\\
& \le x_2-x_1+\mathbb{E}\left[\int_0^{\infty}\left(\rho-\mu_1\right)e^{-\rho t}X_t^1 \dt \right]\\
& = x_2.
\end{align*}
And, for all $0\le\tau_1\le\tau_2,$ $J_0(x_1,x_2,\tau_1,\tau_2,u)$
\begin{align*}
= & ~ \mathbb{E}\big[e^{-\rho\tau_2}\left(\beta_sX_{\tau_2}^1-\beta_bX_{\tau_2}^2\right)\mathbb{I}_{\{\tau_2<\infty\}}\mathbb{I}_{\{u=1\}}\big]-\mathbb{E}\big[e^{-\rho\tau_1}\left(\beta_bX_{\tau_1}^1-\beta_sX_{\tau_1}^2\right)\mathbb{I}_{\{\tau_1<\infty\}}\mathbb{I}_{\{u=1\}}\big]\\
&+ \mathbb{E}\big[e^{-\rho\tau_1}\left(\beta_sX_{\tau_1}^1-\beta_bX_{\tau_1}^2\right)\mathbb{I}_{\{\tau_1<\infty\}}\mathbb{I}_{\{u=-1\}}\big]-\mathbb{E}\big[e^{-\rho\tau_2}\left(\beta_bX_{\tau_2}^1-\beta_sX_{\tau_2}^2\right)\mathbb{I}_{\{\tau_2<\infty\}}\mathbb{I}_{\{u=-1\}}\big]\\
\le &~ x_1-\mathbb{E}\left[x_2 \mathbb{I}_{\{\tau_2<\infty\}}\mathbb{I}_{\{u=1\}}\right]+\mathbb{E}\left[\int_0^{\tau_2}\left(\rho-\mu_2\right)e^{-\rho t}X_t^2\dt \mathbb{I}_{\{\tau_2<\infty\}}\mathbb{I}_{\{u=1\}}\right]\\
&+x_2-\mathbb{E}\left[x_1 \mathbb{I}_{\{\tau_1<\infty\}}\mathbb{I}_{\{u=1\}}\right]+\mathbb{E}\left[\int_0^{\tau_1}\left(\rho-\mu_1\right)e^{-\rho t}X_t^1\dt \mathbb{I}_{\{\tau_1<\infty\}}\mathbb{I}_{\{u=1\}}\right]\\
&+x_1-\mathbb{E}\left[x_2 \mathbb{I}_{\{\tau_1<\infty\}}\mathbb{I}_{\{u=-1\}}\right]+\mathbb{E}\left[\int_0^{\tau_1}\left(\rho-\mu_2\right)e^{-\rho t}X_t^2\dt \mathbb{I}_{\{\tau_1<\infty\}}\mathbb{I}_{\{u=-1\}}\right]\\
&+x_2-\mathbb{E}\left[x_1 \mathbb{I}_{\{\tau_2<\infty\}}\mathbb{I}_{\{u=-1\}}\right]+\mathbb{E}\left[\int_0^{\tau_2}\left(\rho-\mu_1\right)e^{-\rho t}X_t^1\dt \mathbb{I}_{\{\tau_2<\infty\}}\mathbb{I}_{\{u=-1\}}\right]
\end{align*}
Now, 
\begin{align*}
\mathbb{E}\left[\int_0^{\tau_1}\left(\rho-\mu_1\right)e^{-\rho t}X_t^1\dt \mathbb{I}_{\{\tau_1<\infty\}}\mathbb{I}_{\{u=1\}}\right]&\le   \mathbb{E}\left[\int_0^{\tau_1}\left(\rho-\mu_1\right)e^{-\rho t}X_t^1\dt \mathbb{I}_{\{\tau_1<\infty\}}\right]\\
&\le   \mathbb{E}\left[\int_0^{\infty}\left(\rho-\mu_1\right)e^{-\rho t}X_t^1\dt \right]\\
&=  \left(\rho-\mu_1\right)\int_0^{\infty}e^{-\rho t}x_1e^{\mu_1 t}\dt \\
& =  x_1.
\end{align*}
Similarly,
\begin{align*}
\mathbb{E}\left[\int_0^{\tau_2}\left(\rho-\mu_2\right)e^{-\rho t}X_t^2\dt \mathbb{I}_{\{\tau_2<\infty\}}\mathbb{I}_{\{u=1\}}\right] & \le x_2,\\
\mathbb{E}\left[\int_0^{\tau_1}\left(\rho-\mu_2\right)e^{-\rho t}X_t^2\dt \mathbb{I}_{\{\tau_1<\infty\}}\mathbb{I}_{\{u=-1\}}\right] & \le x_2, ~\text{and}\\
\mathbb{E}\left[\int_0^{\tau_2}\left(\rho-\mu_1\right)e^{-\rho t}X_t^1\dt \mathbb{I}_{\{\tau_2<\infty\}}\mathbb{I}_{\{u=-1\}}\right] & \le x_1.
\end{align*}
Thus, $J_0(x_1,x_2,\tau_1,\tau_2,u)\le 4x_1+4x_2.$
\end{proof}

\section{HJB Equations}
In this section, we study the associated HJB equations, which have the form, for $x_1,x_2>0$,
\beq{HJB2}
\left\{\begin{array}{l}
\min\Big\{\rho v_1(x_1,x_2)-\A v_1(x_1,x_2),\
v_1(x_1,x_2)-\Bs  x_1+\Bb  x_2\Big\}=0,\\
\min\Big\{\rho v_{-1}(x_1,x_2)-\A v_{-1}(x_1,x_2),\
v_{-1}(x_1,x_2)+\Bb  x_1-\Bs  x_2\Big\}=0,\\
\min\Big\{\rho v_0(x_1,x_2)-\A v_0(x_1,x_2),\
v_0(x_1,x_2)-v_1(x_1,x_2)+\Bb  x_1-\Bs  x_2,\\
\qquad \qquad \qquad \qquad \qquad 
\qquad \qquad v_0(x_1,x_2)-v_{-1}(x_1,x_2)-\Bs x_1+\Bb x_2\Big\}=0.\\
\end{array}\right.\\
\eeq
As above, the HJB equations can be reduced to an ODE problem by applying the following substitution.
Let $y=x_2/x_1$ and $v_i(x_1,x_2)=x_1w_i(x_2/x_1)$,
for some function $w_i(y)$ and $i=-1,0,1$.
The HJB equations can be given in terms of $y$ and $w_i$ as follows:
\begin{equation}
\label{ode}
\left\{\begin{array}{l}
\min\Big\{\rho w_1(y)-\L w_1(y),\
w_1(y)-\Bs  +\Bb  y\Big\}=0,\\
\min\Big\{\rho w_{-1}(y)-\L w_{-1}(y),\
w_{-1}(y)+\Bb  -\Bs  y\Big\}=0,\\
\min\Big\{\rho w_0(y)-\L w_0(y),\
w_0(y)-w_1(y)+\Bb  -\Bs  y,\ w_0(y)-w_{-1}(y)-\Bs+\Bb y \Big\}=0,\\
\end{array}\right.
\end{equation}
As above, $(\rho-\mathcal{L})w_i(y)=0$, $i=-1,0,1$ can be rewritten as Euler-type equations with solutions of the form $y^{\delta}$ for $\delta_1$, $\delta_2$ as in (\ref{delta1}), (\ref{delta2}).

Now, we would like to open pairs position $\textbf{Z}$ when the price of $\textbf{S}^2$ is large relative to the price of $\textbf{S}^1$ ($k_3$ and $k_4$) and close pairs position $\textbf{Z}$ when the price of $\textbf{S}^2$ is small relative to the price of  $\textbf{S}^1$ ($k_1$ and $k_2$).  Additionally, we would be more willing to open pairs position $\textbf{Z}$ when the net position is short than when the net position is flat, since when the net position is short we experience the risk of holding one share of $\textbf{S}^2$ while borrowing one share of  $\textbf{S}^1$.  Similarly, we would be more willing to close pairs position $\textbf{Z}$ when the net position is long than when the net position is flat, since when the net position is long we experience the risk of borrowing one share of $\textbf{S}^2$ while holding one share of  $\textbf{S}^1$.  This suggests that we should expect $k_1\le k_2\le k_3\le k_4.$
\begin{figure}
\includegraphics[width=0.75\textwidth]{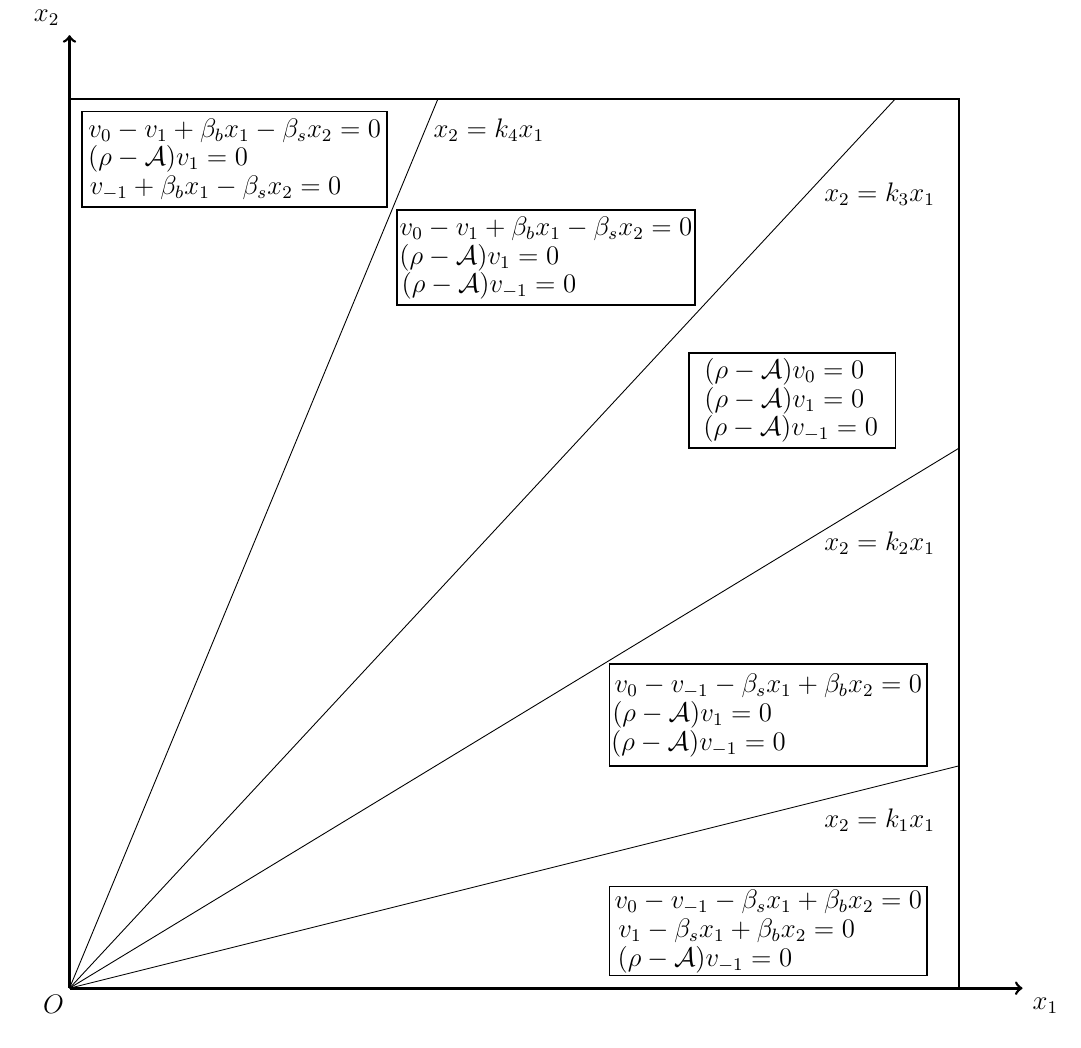}
\caption{Thresholds for buying and selling regions}
\end{figure}

Suppose we can find $k_1$, $k_4$, $k_2$ and $k_3$ with $k_1<k_2<k_3<k_4$ so that the first equation
\[\min\Big\{\rho w_1(y)-\L w_1(y),\
w_1(y)-\Bs  +\Bb  y\Big\}=0\]
has solution 
\beq{}
w_1(y)=\begin{cases} \beta_s-\beta_b y,\/\ \/\ &\/\ \text{if~} 0\leq y \leq k_1\\
C_2 y^{\delta_2},\/\ \/\ &\/\ \text{if~} y\geq k_1.
\end{cases}
\eeq
Then the smooth-fitting conditions yield
\[ \beta_s-\beta_b k_1=C_2 k_1^{\delta_2} \quad \text{and}\quad -\beta_b=C_2\delta_2 k_1^{\delta_2-1}\]
This will imply
\begin{equation}
\label{k_1}
k_1=\frac{-\delta_2}{1-\delta_2} \cdot \frac{\beta_s}{\beta_b}
\end{equation}
and
\beq{C_2}
C_2=\frac{\beta_b}{-\delta_2}\cdot k_1^{1-\delta_2}=\left(\frac{-\delta_2}{\beta_b}\right)^{-\delta_2} \left(\frac{\beta_s}{1-\delta_2}\right)^{1-\delta_2}.
\eeq

\par The second equation
\[\min\Big\{\rho w_{-1}(y)-\L w_{-1}(y),\
w_{-1}(y)+\Bb  -\Bs  y\Big\}=0\]
has solution
\beq{}
w_{-1} (y)=\begin{cases} C_1 y^{\delta_1},\/\ \/\ &\/\ \text{if~} 0\leq y \leq k_4\\
\beta_s y-\beta_b,\/\ \/\ &\/\ \text{if~} y\geq k_4.
\end{cases}
\eeq
Then the smooth-fitting conditions yield
\[C_1 k_4^{\delta_1}=\beta_s k_4-\beta_b  \quad \text{and}\quad C_1\delta_1 k_4^{\delta_1-1}=\beta_s\]
This will imply
\begin{equation}
\label{k_4}
k_4=\frac{\delta_1}{\delta_1-1}\cdot \frac{\beta_b}{\beta_s}
\end{equation}
and
\beq{C_1}
C_1=\frac{\beta_s}{\delta_1}\cdot k_4^{1-\delta_1}=\left(\frac{\beta_s}{\delta_1}\right)^{\delta_1}\left(\frac{\delta_1-1}{\beta_b}\right)^{\delta_1-1}.
\eeq

\par The third equation
\[\min\Big\{\rho w_0(y)-\L w_0(y),\
w_0(y)-w_1(y)+\Bb  -\Bs  y,\ w_0(y)-w_{-1}(y)-\Bs+\Bb y \Big\}=0\]
has solution
\beq{}
w_0(y)=\begin{cases} C_1y^{\delta_1} +\Bs -\Bb y,\/\ \/\ &\/\ \text{if~} 0\leq y \leq k_2\\
B_1y^{\delta_1}+B_2 y^{\delta_2},\/\ \/\ &\/\ \text{if~} k_2\leq y\leq k_3\\
C_2y^{\delta_2}-\Bb +\Bs y,\/\ \/\ &\/\ \text{if~} y\geq k_3.
\end{cases}
\eeq

Then the smooth-fitting conditions yield
\begin{align*}
C_1k_2^{\delta_1} +\Bs -\Bb k_2&=B_1k_2^{\delta_1}+B_2 k_2^{\delta_2} \\
C_1\delta_1 k_2^{\delta_1-1 } -\Bb &=B_1\delta_1 k_2^{\delta_1-1}+B_2 \delta_2 k_2^{\delta_2-1}\\
B_1k_3^{\delta_1}+B_2 k_3^{\delta_2} &=C_2 k_3^{\delta_2}-\Bb +\Bs k_3\\
B_1\delta_1 k_3^{\delta_1-1}+B_2 \delta_2 k_3^{\delta_2-1} &=C_2\delta_2 k_3^{\delta_2-1}+\Bs 
\end{align*}
There are four equations and four parameters, $B_1$, $B_2$, $k_2$, and $k_4$, that need to be found.
These equations can be written in the matrix form:
\[ \begin{pmatrix}
k_2^{\delta_1} & k_2^{\delta_2}\\
\delta_1 k_2^{\delta_1-1} & \delta_2 k_2^{\delta_2-1}
\end{pmatrix}
\begin{pmatrix}
B_1-C_1\\
B_2
\end{pmatrix}
=\begin{pmatrix}
1 & -k_2\\
0 & -1
\end{pmatrix}
\begin{pmatrix}
\Bs\\ \Bb
\end{pmatrix}
\]
and
\[ \begin{pmatrix}
k_3^{\delta_1} & k_3^{\delta_2}\\
\delta_1 k_3^{\delta_1-1} & \delta_2 k_3^{\delta_2-1}
\end{pmatrix}
\begin{pmatrix}
B_1\\
B_2-C_2
\end{pmatrix}
=\begin{pmatrix}
k_3 & -1\\
1 & 0
\end{pmatrix}
\begin{pmatrix}
\Bs\\ \Bb
\end{pmatrix}
.\]
We introduce a new matrix
\[\Phi(r)=\begin{pmatrix}
r^{\delta_1} & r^{\delta_2}\\
\delta_1 r^{\delta_1-1} & \delta_2 r^{\delta_2-1}
\end{pmatrix}\quad \text{and its inverse}\quad
\Phi(r)^{-1}=\frac{1}{\delta_1-\delta_2}
\begin{pmatrix}
-\delta_2 r^{-\delta_1} & r^{1-\delta_1}\\
\delta_1 r^{-\delta_2} & -r^{1-\delta_2}
\end{pmatrix}
\]
for $r\ne0$.
Returning to the smooth-fit conditions above, we have
\[\begin{pmatrix}
B_1-C_1\\
B_2
\end{pmatrix}=\Phi(k_2)^{-1}\begin{pmatrix}
1 & -k_2\\
0 & -1
\end{pmatrix}
\begin{pmatrix}
\Bs\\ \Bb
\end{pmatrix}
\]
and
\[\begin{pmatrix}
B_1\\
B_2-C_2
\end{pmatrix}=\Phi(k_3)^{-1}\begin{pmatrix}
k_3 & -1\\
1 & 0
\end{pmatrix}
\begin{pmatrix}
\Bs\\ \Bb
\end{pmatrix}
\]
This implies
\[\begin{pmatrix}
B_1\\
B_2
\end{pmatrix}=\begin{pmatrix}
C_1\\
0
\end{pmatrix}+\Phi(k_2)^{-1}\begin{pmatrix}
1 & -k_2\\
0 & -1
\end{pmatrix}
\begin{pmatrix}
\Bs\\ \Bb
\end{pmatrix}
=\begin{pmatrix}
0\\
C_2
\end{pmatrix}+\Phi(k_3)^{-1}\begin{pmatrix}
k_3 & -1\\
1 & 0
\end{pmatrix}
\begin{pmatrix}
\Bs\\ \Bb
\end{pmatrix}.\]
The second equality yields two equations of $k_2$ and $k_3$, we can rewrite it as
\[\left[\Phi(k_3)^{-1}\begin{pmatrix}
k_3 & -1\\
1 & 0
\end{pmatrix}- \Phi(k_2)^{-1}\begin{pmatrix}
1 & -k_2\\
0 & -1
\end{pmatrix}\right]
\begin{pmatrix}
\Bs\\ \Bb
\end{pmatrix}
=\begin{pmatrix}
C_1\\
-C_2
\end{pmatrix}
\]
The matrix in $[\cdot]$ is 
\[\frac{1}{\delta_1-\delta_2}
\begin{pmatrix}
(1-\delta_2)k_3^{1-\delta_1}+\delta_2k_2^{-\delta_1} & \delta_2 k_3^{-\delta_1}+(1-\delta_2)k_2^{1-\delta_1}\\
-(1-\delta_1)k_3^{1-\delta_2}-\delta_1k_2^{-\delta_2} & -\delta_1k_3^{-\delta_2}-(1-\delta_1)k_2^{1-\delta_2}
\end{pmatrix}
\]
The two equations involving $k_2$ and $k_3$ are
\[\frac{1}{\delta_1-\delta_2}
\begin{pmatrix}
(1-\delta_2)k_3^{1-\delta_1}+\delta_2k_2^{-\delta_1} & \delta_2 k_3^{-\delta_1}+(1-\delta_2)k_2^{1-\delta_1}\\
(1-\delta_1)k_3^{1-\delta_2}+\delta_1k_2^{-\delta_2} & \delta_1k_3^{-\delta_2}+(1-\delta_1)k_2^{1-\delta_2}
\end{pmatrix}
\begin{pmatrix}
\Bs\\ \Bb
\end{pmatrix}
=\begin{pmatrix}
C_1\\
C_2
\end{pmatrix}
\]
Recall that
\[ C_1=\frac{\beta_s}{\delta_1} \cdot k_4^{1-\delta_1}=\left(\frac{\beta_s}{\delta_1}\right)^{\delta_1}\left(\frac{\delta_1-1}{\beta_b}\right)^{\delta_1-1}
\ \text{and}\ C_2=\frac{\beta_b}{-\delta_2} \cdot k_1^{1-\delta_2}=\left(\frac{-\delta_2}{\beta_b}\right)^{-\delta_2} \left(\frac{\beta_s}{1-\delta_2}\right)^{1-\delta_2}.\]
The system of equations for $k_2$ and $k_3$ are
\begin{align*}
\frac{(1-\delta_2)k_3^{1-\delta_1}+\delta_2k_2^{-\delta_1} }{\delta_1-\delta_2} \Bs
+\frac{\delta_2 k_3^{-\delta_1}+(1-\delta_2)k_2^{1-\delta_1} }{\delta_1-\delta_2} \Bb & =
\frac{\beta_s}{\delta_1}\cdot k_4^{1-\delta_1}\\
\frac{(1-\delta_1)k_3^{1-\delta_2}+\delta_1k_2^{-\delta_2}}{\delta_1-\delta_2} \Bs
+\frac{\delta_1k_3^{-\delta_2}+(1-\delta_1)k_2^{1-\delta_2}}{\delta_1-\delta_2} \Bb & =
\frac{\beta_b}{-\delta_2} \cdot k_1^{1-\delta_2}.
\end{align*}
We are looking for solutions $(k_2,k_3)$ in the triangular  region 
\[ T=\{(r,s):\ k_1\le r<s\le k_4\} \subset \R_{+}^2.\]
Let $\gamma=\dfrac{\beta_b}{\beta_s}$. Then we can reduce the system to
\begin{align}
F_1(k_2,k_3):=\frac{(1-\delta_2)k_3^{1-\delta_1}+\delta_2k_2^{-\delta_1} }{\delta_1-\delta_2} 
+\frac{\delta_2 k_3^{-\delta_1}+(1-\delta_2)k_2^{1-\delta_1} }{\delta_1-\delta_2} \gamma
- \frac{k_4^{1-\delta_1}}{\delta_1} & =0\label{F1}
\\
F_2(k_2,k_3):=\frac{(1-\delta_1)k_3^{1-\delta_2}+\delta_1k_2^{-\delta_2}}{\delta_1-\delta_2} 
+\frac{\delta_1k_3^{-\delta_2}+(1-\delta_1)k_2^{1-\delta_2}}{\delta_1-\delta_2} \gamma 
-\frac{\gamma k_1^{1-\delta_2}}{-\delta_2} 
& =0\label{F2}.
\end{align}
\begin{figure}[ht]
\centering\scalebox{1}{\includegraphics{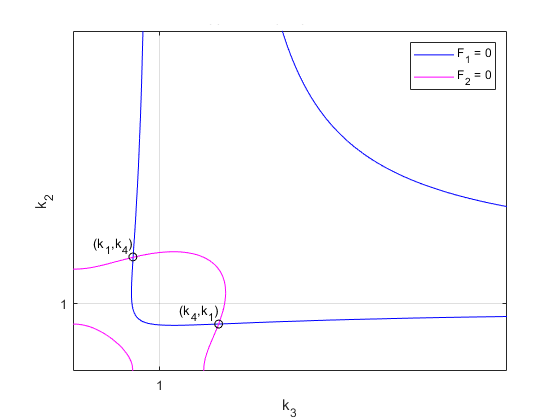}}
\caption{Numerical solution to system of equations in (\ref{F1}) and (\ref{F2}).}
\end{figure}

Note that $(k_1,k_4)$ is a solution to the system, since:

\begin{align*}
F_1(k_1,k_4)& =\frac{(1-\delta_2)k_4^{1-\delta_1}+\delta_2k_1^{-\delta_1} }{\delta_1-\delta_2} 
+\frac{\delta_2 k_1^{-\delta_1}+(1-\delta_2)k_4^{1-\delta_1} }{\delta_1-\delta_2} \gamma
- \frac{k_4^{1-\delta_1}}{\delta_1}\\
&=\frac{k_4^{-\delta_1}}{\delta_1-\delta_2}\left[\frac{(1-\delta_2) \delta_1}{\delta_1-1}\gamma-\frac{\delta_1-\delta_2}{\delta_1-1}\gamma +\delta_2\gamma\right]+\frac{k_1^{-\delta_1}}{\delta_1-\delta_2}\left[\delta_2+(-\delta_2)\right]\\
&=0
\end{align*}
and
\begin{align*}
F_2(k_1,k_4)& =\frac{(1-\delta_1)k_4^{1-\delta_2}+\delta_1k_1^{-\delta_2}}{\delta_1-\delta_2} 
+\frac{\delta_1k_4^{-\delta_2}+(1-\delta_1)k_1^{1-\delta_2}}{\delta_1-\delta_2} \gamma 
-\frac{\gamma k_1^{1-\delta_2}}{-\delta_2} \\
&=\frac{k_4^{-\delta_2}}{\delta_1-\delta_2} \left[-\delta_1\gamma+\delta_1\gamma\right]
+\frac{k_1^{-\delta_2}}{\delta_1-\delta_2}\left[\delta_1+\frac{(1-\delta_1)(-\delta_2)}{1-\delta_2}-\frac{\delta_1-\delta_2}{1-\delta_2}\right]\\
&=0.
\end{align*}

Now, recall that the smooth-fit conditions for $w_0$ can be written as:\\
$\left\{\begin{array}{l}
(B_1-C_1)k_2^{\delta_1}+B_2k_2^{\delta_2}=\beta_s-\beta_bk_2\\

(B_1-C_1)\delta_1k_2^{\delta_1-1}+B_2\delta_2k_2^{\delta_2-1}=-\beta_b
\end{array}\right.$\\
\\
and\\
\\
$\left\{\begin{array}{l}
B_1k_3^{\delta_1}+(B_2-C_2)k_3^{\delta_2}=\beta_sk_3-\beta_b\\

B_1\delta_1k_3^{\delta_1-1}+(B_2-C_2)\delta_2k_3^{\delta_2-1}=\beta_s
\end{array}\right.$\\
\\
From these we obtain:
\begin{align}
&\dfrac{(B_1-C_1)k_2^{\delta_1}}{(B_1-C_1)\delta_1k_2^{\delta_1-1}} =\dfrac{\beta_s-\beta_bk_2-B_2k_2^{\delta_2}}{-\beta_b-B_2\delta_2k_2^{\delta_2-1}}\notag\\
\implies & \dfrac{k_2}{\delta_1}=\dfrac{\beta_s-\beta_bk_2-B_2k_2^{\delta_2}}{-\beta_b-B_2\delta_2k_2^{\delta_2-1}}\notag\\
\implies & \beta_s\delta_1-\beta_b\delta_1k_2-B_2\delta_1k_2^{\delta_2}=-\beta_bk_2-B_2\delta_2k_2^{\delta_2}\notag\\
\implies & B_2=\dfrac{\beta_s\delta_1k_2^{-\delta_2}-\beta_b(\delta_1-1)k_2^{1-\delta_2}}{\delta_1-\delta_2}
\end{align}
Also
\begin{align}
&\dfrac{B_2k_2^{\delta_2}}{B_2\delta_2k_2^{\delta_2-1}}=\dfrac{\beta_s-\beta_bk_2-(B_1-C_1)k_2^{\delta_1}}{-\beta_b-(B_1-C_1)\delta_1k_2^{\delta_1-1}}\notag\\
\implies & \dfrac{k_2}{\delta_2}=\dfrac{\beta_s-\beta_bk_2-(B_1-C_1)k_2^{\delta_1}}{-\beta_b-(B_1-C_1)\delta_1k_2^{\delta_1-1}}\notag\\
\implies & \beta_s\delta_2-\beta_b\delta_2k_2-(B_1-C_1)\delta_2k_2^{\delta_1}=-\beta_bk_2-(B_1-C_1)\delta_1k_2^{\delta_1}\notag\\
\implies & B_1-C_1=\dfrac{\beta_s(-\delta_2)k_2^{-\delta_1}-\beta_b(1-\delta_2)k_2^{1-\delta_1}}{\delta_1-\delta_2}
\end{align}

\begin{align}
&\dfrac{B_1k_3^{\delta_1}}{B_1\delta_1k_3^{\delta_1-1}}=\dfrac{\beta_sk_3-\beta_b-(B_2-C_2)k_3^{\delta_2}}{\beta_s-(B_2-C_2)\delta_2k_3^{\delta_2-1}}\notag\\
\implies & \dfrac{k_3}{\delta_1}=\dfrac{\beta_sk_3-\beta_b-(B_2-C_2)k_3^{\delta_2}}{\beta_s-(B_2-C_2)\delta_2k_3^{\delta_2-1}}\notag\\
\implies & \beta_s\delta_1k_3-\beta_b\delta_1-(B_2-C_2)\delta_1k_3^{\delta_2}=\beta_sk_3-(B_2-C_2)\delta_2k_3^{\delta_2}\notag\\
\implies & B_2-C_2=\dfrac{\beta_s(\delta_1-1)k_3^{1-\delta_2}-\beta_b\delta_1k_3^{-\delta_2}}{\delta_1-\delta_2}
\end{align}
and
\begin{align}
&\dfrac{(B_2-C_2)k_3^{\delta_2}}{(B_2-C_2)\delta_2k_3^{\delta_2-1}}=\dfrac{\beta_sk_3-\beta_b-B_1k_3^{\delta_1}}{\beta_s-B_1\delta_1k_3^{\delta_1-1}}\notag\\
\implies & \dfrac{k_3}{\delta_2}=\dfrac{\beta_sk_3-\beta_b-B_1k_3^{\delta_1}}{\beta_s-B_1\delta_1k_3^{\delta_1-1}}\notag\\
\implies & \beta_s\delta_2k_3-\beta_b\delta_2-B_1\delta_2k_3^{\delta_1}=\beta_sk_3-B_1\delta_1k_3^{\delta_1}\notag\\
\implies & B_1=\dfrac{\beta_s(1-\delta_2)k_3^{1-\delta_1}-\beta_b(-\delta_2)k_3^{-\delta_1}}{\delta_1-\delta_2}
\end{align}
Note then that since $k_2=k_1$, we have
\begin{align*}
&k_2=\dfrac{\beta_s}{\beta_b}\cdot\dfrac{-\delta_2}{1-\delta_2}\\
\implies & \beta_s(-\delta_2)=\beta_b(1-\delta_2)k_2 \\
\implies & \beta_s(-\delta_2)k_2^{-\delta_1}=\beta_b(1-\delta_2)k_2^{1-\delta_1} \\
\implies & \dfrac{\beta_s(-\delta_2)k_2^{-\delta_1}-\beta_b(1-\delta_2)k_2^{1-\delta_1}}{\delta_1-\delta_2}=0 \\
\implies & B_1-C_1=0\\
\implies & B_1=C_1
\end{align*}
Also, since $k_3=k_4$, we have
\begin{align*}
&k_3=\dfrac{\beta_b}{\beta_s}\cdot\dfrac{\delta_1}{\delta_1-1}\\
\implies & \beta_b\delta_1=\beta_s(\delta_1-1)k_3 \\
\implies & \beta_b\delta_1k_3^{-\delta_2}=\beta_s(\delta_1-1)k_3^{1-\delta_2} \\
\implies & \dfrac{\beta_s(\delta_1-1)k_3^{1-\delta_2}-\beta_b\delta_1k_3^{-\delta_2}}{\delta_1-\delta_2}=0 \\
\implies & B_2-C_2=0\\
\implies & B_2=C_2
\end{align*}
Hence, in this case
\beq{}
w_0(y)=\begin{cases} C_1y^{\delta_1} +\Bs -\Bb y & 0\leq y \leq k_1,\\
C_1y^{\delta_1}+C_2 y^{\delta_2} & k_1\leq y\leq k_4,\\
C_2y^{\delta_2}-\Bb +\Bs y & y\geq k_4.
\end{cases}
\eeq
Let us relabel these threshholds as 
\begin{align}
& k^*_1=k_1=k_2\label{k*1}\\
& k^*_2=k_3=k_4.\label{k*2}
\end{align} Then we have the following.
\begin{thmm}
Let $\delta_i$ be given by (\ref{delta1}), (\ref{delta2}) and $k_i$ be given by (\ref{k*1}),~(\ref{k*2}).  Then the following functions $w_1$, $w_{-1}$, and $w_0$ satisfy the HJB equations (\ref{ode}):\\
$w_1(y)=\begin{cases}
\beta_s-\beta_b y, &\text{if~}0\le y\le k^*_1, \\
\left(-\dfrac{\delta_2}{\beta_b}\right)^{-\delta_2}\left(\dfrac{\beta_s}{1-\delta_2}\right)^{1-\delta_2}y^{\delta_2},&\text{if~} y\ge k^*_1. \\
\end{cases}$\\
\\
$w_{-1}(y)=\begin{cases}
\left(\dfrac{\beta_s}{\delta_1}\right)^{\delta_1}\left(\dfrac{\delta_1-1}{\beta_b}\right)^{\delta_1-1}y^{\delta_1},& \text{if~}0\le y\le k^*_2 ,\\
\beta_s y-\beta_b, &\text{if~} y\ge k^*_2. \\
\end{cases}$\\
\\
$w_0(y)=\begin{cases}
\left(\dfrac{\beta_s}{\delta_1}\right)^{\delta_1}\left(\dfrac{\delta_1-1}{\beta_b}\right)^{\delta_1-1}y^{\delta_1}+\beta_s-\beta_b y, &\text{if~} 0\le y\le k^*_1,\\
\left(\dfrac{\beta_s}{\delta_1}\right)^{\delta_1}\left(\dfrac{\delta_1-1}{\beta_b}\right)^{\delta_1-1}y^{\delta_1}+\left(-\dfrac{\delta_2}{\beta_b}\right)^{-\delta_2}\left(\dfrac{\beta_s}{1-\delta_2}\right)^{1-\delta_2}y^{\delta_2}, &\text{if~} k^*_1\le y\le k^*_2 ,\\
\left(-\dfrac{\delta_2}{\beta_b}\right)^{-\delta_2}\left(\dfrac{\beta_s}{1-\delta_2}\right)^{1-\delta_2}y^{\delta_2}-\beta_b+\beta_s y,& \text{if~}y\ge k^*_2. \\
\end{cases}$
\end{thmm}
\begin{proof}We divide the first quadrant of the plane into 3 regions, 
\[\Gamma_1: 0<y\le k^*_1,~~~\/\ \/\ \/\ \\
\Gamma_2: k^*_1<y\le k^*_2,~~~\/\ \/\ \/\ \\
\Gamma_3: k^*_2<y
\]
Thus, to establish that we have found a solution to the HJB equations, we must establish the following list of variational inequalities:\\
\\
\[\begin{cases} (\rho-\mathcal{L})w_1(y)\ge 0 & \text{on }\Gamma_1\\
w_1(y)-\beta_s+\beta_b y\ge 0 & \text{on }\Gamma_2\cup\Gamma_3\\
w_{-1}(y)+\beta_b-\beta_s y\ge 0 & \text{on }\Gamma_1\cup\Gamma_2\\
(\rho-\mathcal{L})w_{-1}(y)\ge 0 & \text{on }\Gamma_3\\
(\rho-\mathcal{L})w_0(y)\ge 0 & \text{on }\Gamma_1\cup\Gamma_3\\
w_0(y)-w_1(y)+\beta_b-\beta_s y\ge 0 & \text{on }\Gamma_1\cup\Gamma_2\\
w_0(y)-w_{-1}(y)-\beta_s+\beta_b y\ge 0& \text{on }\Gamma_2\cup\Gamma_3\\
\end{cases}
\]
On $\Gamma_1$, 
\begin{align*}
(\rho-\mathcal{L})w_1(y)&=(\rho-\mathcal{L})(\beta_s-\beta_b y)\\
&=\rho\beta_s-\mu_1\beta_s+\mu_1\beta_by+(\mu_2-\mu_1)\beta_by-\rho\beta_by\\
&=(\rho-\mu_1)\beta_s-(\rho-\mu_2)\beta_by
\end{align*}
Hence, 
\begin{align*}
(\rho-\mathcal{L})w_1(y)\ge0&\iff(\rho-\mu_1)\beta_s\ge(\rho-\mu_2)\beta_by\\
&\iff y\le\dfrac{\rho-\mu_1}{\rho-\mu_2}\cdot\dfrac{\beta_s}{\beta_b}\\
&\iff y\le\dfrac{\delta_1}{\delta_1-1}\cdot\dfrac{-\delta_2}{1-\delta_2}\cdot\dfrac{\beta_s}{\beta_b}\\
&\iff y\le\dfrac{\delta_1}{\delta_1-1}\cdot k^*_1,
\end{align*}
which holds, since $y\le k^*_1\le \dfrac{\delta_1}{\delta_1-1}\cdot k^*_1$.\\
On $\Gamma_2\cup\Gamma_3$, 
\[w_1(y)-\beta_s+\beta_by\ge0\iff C_2y^{\delta_2}-\beta_s+\beta_by\ge0.\]
Let $f(y)=C_2y^{\delta_2}-\beta_s+\beta_by$.  Then
\begin{align*}
f'(y)\ge0&\iff C_2\delta_2y^{\delta_2-1}+\beta_b\ge0\\
&\iff C_2(-\delta_2)y^{\delta_2-1}\le \beta_b\\
&\iff y^{\delta_2-1}\le \dfrac{\beta_b}{C_2(-\delta_2)}=k_1^{\delta_2-1}\\
&\iff y^{1-\delta_2}\ge {k^*_1}^{1-\delta_2}\\
&\iff y\ge k^*_1,
\end{align*}
which clearly holds.  Hence $f(y)$ is increasing for $y\ge k^*_1.$  Since $f(k^*_1)=0$, it must be that $w_1(y)-\beta_s+\beta_by\ge0$ on $\Gamma_2\cup\Gamma_3$.\\
On $\Gamma_1\cup\Gamma_2$, 
\[w_{-1}(y)+\beta_b-\beta_sy\ge0\iff C_1y^{\delta_1}+\beta_b-\beta_sy\ge0.\]
Let $g(y)=C_1y^{\delta_1}+\beta_b-\beta_sy$.  Then
\begin{align*}
g'(y)\le0&\iff C_1\delta_1y^{\delta_1-1}-\beta_s\le0\\
&\iff C_1\delta_1y^{\delta_1-1}\le \beta_s\\
&\iff y^{\delta_1-1}\le \dfrac{\beta_s}{C_1\delta_1}={k^*_2}^{\delta_1-1}\\
&\iff y\le k^*_2,
\end{align*}
which clearly holds.  Hence $g(y)$ is decreasing for $y\le k^*_2.$  Since $g(k^*_2)=0$, it must be that $w_{-1}(y)+\beta_b-\beta_sy\ge0$ on $\Gamma_1\cup\Gamma_2$.\\
On $\Gamma_3$, 
\begin{align*}
(\rho-\mathcal{L})w_{-1}(y)&=(\rho-\mathcal{L})(\beta_sy-\beta_b )\\
&=\rho\beta_sy-\mu_2\beta_sy+\mu_1\beta_b-\rho\beta_b\\
&=(\rho-\mu_2)\beta_sy-(\rho-\mu_1)\beta_b
\end{align*}
Hence, 
\begin{align*}
(\rho-\mathcal{L})w_{-1}(y)\ge0&\iff(\rho-\mu_2)\beta_sy\ge(\rho-\mu_1)\beta_b\\
&\iff y\ge\dfrac{\rho-\mu_1}{\rho-\mu_2}\cdot\dfrac{\beta_b}{\beta_s}\\
&\iff y\ge\dfrac{-\delta_2}{1-\delta_2}\cdot\dfrac{\delta_1}{\delta_1-1}\cdot\dfrac{\beta_b}{\beta_s}\\
&\iff y\ge\dfrac{-\delta_2}{1-\delta_2}\cdot k^*_2,
\end{align*}
which holds, since $y\ge k^*_2\ge \dfrac{-\delta_2}{1-\delta_2}\cdot k^*_2$.\\
On $\Gamma_1$, 
\begin{align*}
(\rho-\mathcal{L})w_0(y)&=(\rho-\mathcal{L})(w_{-1}(y)+w_1(y))\\
&=0+(\rho-\mathcal{L})w_1(y),
\end{align*}
and we have already established that $(\rho-\mathcal{L})w_1(y)\ge 0$ on $\Gamma_1$.\\
On $\Gamma_3$, 
\begin{align*}
(\rho-\mathcal{L})w_0(y)&=(\rho-\mathcal{L})(w_1(y)+w_{-1}(y))\\
&=(\rho-\mathcal{L})w_1(y)+(\rho-\mathcal{L})w_{-1}(y)\\
&=0+(\rho-\mathcal{L})w_{-1}(y),
\end{align*}
and we have already established that $(\rho-\mathcal{L})w_{-1}(y)\ge 0$ on $\Gamma_3$.\\
On $\Gamma_1$, 
\begin{align*}
w_0(y)-w_1(y)+\beta_b-\beta_sy&=C_1y^{\delta_1}+\beta_s-\beta_by-\beta_s+\beta_by+\beta_b-\beta_sy\\
&=C_1y^{\delta_1}+\beta_b-\beta_sy,
\end{align*}
and we have already established that $C_1y^{\delta_1}+\beta_b-\beta_sy\ge 0$ on $\Gamma_1$.\\
On $\Gamma_2$, 
\begin{align*}
w_0(y)-w_1(y)+\beta_b-\beta_sy&=C_1y^{\delta_1}+C_2y^{\delta_2}-C_2y^{\delta_2}+\beta_b-\beta_sy\\
&=C_1y^{\delta_1}+\beta_b-\beta_sy,
\end{align*}
and we have already established that $C_1y^{\delta_1}+\beta_b-\beta_sy\ge 0$ on $\Gamma_2$.\\
On $\Gamma_2$, 
\begin{align*}
w_0(y)-w_{-1}(y)-\beta_s+\beta_by&=C_1y^{\delta_1}+C_2y^{\delta_2}-C_1y^{\delta_1}-\beta_s+\beta_by\\
&=C_2y^{\delta_2}-\beta_s+\beta_by,
\end{align*}
and we have already established that $C_2y^{\delta_2}-\beta_s+\beta_by\ge 0$ on $\Gamma_2$.\\
On $\Gamma_3$, 
\begin{align*}
w_0(y)-w_{-1}(y)-\beta_s+\beta_by&=C_2y^{\delta_2}-\beta_b+\beta_sy-\beta_sy+\beta_b-\beta_s+\beta_by\\
&=C_2y^{\delta_2}-\beta_s+\beta_by,
\end{align*}
and we have already established that $C_2y^{\delta_2}-\beta_s+\beta_by\ge 0$ on $\Gamma_3$.\\
\end{proof}
\section{Verification Theorem}
\begin{thmm}
We have $v_i(x_1,x_2)=x_1w_i\left(\frac{x_2}{x_1}\right)=V_i(x_1,x_2)$, $i=-1,0,1$.  If initially $i=-1$, let $\tau_0^*=\inf{\left\{t\ge0:\left(X_t^1,X_t^2\right)\in \Gamma_3 \right\}}$.  If initially $i=1$, let $\tau_0^*=\inf{\left\{t\ge0:\left(X_t^1,X_t^2\right)\in \Gamma_1 \right\}}$.  Finally, if initially $i=0$, let $\tau_1^*=\inf{\left\{t\ge0:\left(X_t^1,X_t^2\right)\notin \Gamma_2 \right\}}$.  If $\left(X_{\tau_1^*}^1,X_{\tau_1^*}^2\right)\in \Gamma_1$, then $u^*=-1$ and $\tau_2^*=\inf{\left\{t\ge\tau_1^*:\left(X_t^1,X_t^2\right)\in \Gamma_3 \right\}}$.  Otherwise, if $\left(X_{\tau_1^*}^1,X_{\tau_1^*}^2\right)\in \Gamma_3$, then $u^*=1$ and $\tau_2^*=\inf{\left\{t\ge\tau_1^*:\left(X_t^1,X_t^2\right)\in \Gamma_1 \right\}}$.
\end{thmm}
\begin{proof}
Given $(\rho-\mathcal{A}) v_i(x_1,x_2)\ge 0$, $i=-1,0,1$, and applying Dynkin's formula and Fatou's Lemma as in  $\O{}$ksendal \cite{Oks}, we have for any stopping times $0\le\tau_1\le\tau_2$, almost surely, $\mathbb{E}e^{-\rho\tau_1}v_i\left(X_{\tau_1}^1, X_{\tau_1}^2\right)\ge\mathbb{E}e^{-\rho\tau_2}v_i\left(X_{\tau_2}^1, X_{\tau_2}^2\right).$\\
Hence, we have 
\begin{align*}
v_0(x_1,x_2)\ge&~\mathbb{E}\left[e^{-\rho\tau_1}v_0\left(X_{\tau_1}^1, X_{\tau_1}^2\right)\right]\\
\ge & ~ \mathbb{E}\left[e^{-\rho\tau_1}\left(v_1\left(X_{\tau_1}^1, X_{\tau_1}^2\right)-\beta_bX_{\tau_1}^1+\beta_sX_{\tau_1}^2\right)\mathbb{I}_{\{\tau_1<\infty\}}\mathbb{I}_{\{u=1\}}\right]\\
& ~+ \mathbb{E}\left[e^{-\rho\tau_1}\left(v_{-1}\left(X_{\tau_1}^1, X_{\tau_1}^2\right)+\beta_sX_{\tau_1}^1-\beta_bX_{\tau_1}^2\right)\mathbb{I}_{\{\tau_1<\infty\}}\mathbb{I}_{\{u=-1\}}\right]\\
\ge&~\mathbb{E}\left[e^{-\rho\tau_2}\left(\beta_sX_{\tau_2}^1-\beta_bX_{\tau_2}^2\right)\mathbb{I}_{\{\tau_2<\infty\}}\mathbb{I}_{\{u=1\}}\right]- \mathbb{E}\left[e^{-\rho\tau_2}\left(\beta_bX_{\tau_2}^1 -\beta_sX_{\tau_2}^2\right)\mathbb{I}_{\{\tau_2<\infty\}}\mathbb{I}_{\{u=-1\}}\right]\\
&~+ \mathbb{E}\left[e^{-\rho\tau_1}\left(\beta_sX_{\tau_1}^1-\beta_bX_{\tau_1}^2\right)\mathbb{I}_{\{\tau_1<\infty\}}\mathbb{I}_{\{u=-1\}}\right]-\mathbb{E}\left[e^{-\rho\tau_1}\left(\beta_bX_{\tau_1}^1-\beta_sX_{\tau_1}^2\right)\mathbb{I}_{\{\tau_1<\infty\}}\mathbb{I}_{\{u=1\}}\right]\\
=&~J_0\left(x_1,x_2,\tau_1,\tau_2,u\right),
\end{align*}
for all $0\le\tau_1\le\tau_2$.  This implies $v_0\left(x_1,x_2\right)\ge V_0\left(x_1,x_2\right)$.  Also,
\begin{align*}
v_{1}(x_1,x_2)\ge&~\mathbb{E}\left[e^{-\rho\tau_0}v_{1}\left(X_{\tau_0}^1, X_{\tau_0}^2\right)\right]\\
\ge&~\mathbb{E}\left[e^{-\rho\tau_0}v_{1}\left(X_{\tau_0}^1, X_{\tau_0}^2\right)\mathbb{I}_{\{\tau_0<\infty\}}\right]\\
=&~J_1\left(x_1,x_2,\tau_0\right),\\
v_{-1}(x_1,x_2)\ge&~\mathbb{E}\left[e^{-\rho\tau_0}v_{-1}\left(X_{\tau_0}^1, X_{\tau_0}^2\right)\right]\\
&\ge\mathbb{E}\left[e^{-\rho\tau_0}\left(\beta_s X_{\tau_0}^1-\beta_bX_{\tau_0}^2\right)\mathbb{I}_{\{\tau_0<\infty\}}\right]\\
=&~J_{-1}\left(x_1,x_2,\tau_0\right),
\end{align*}
for all $0\le\tau_0$.  Hence $v_{1}\left(x_1,x_2\right)\ge V_{1}\left(x_1,x_2\right)$ and $v_{-1}\left(x_1,x_2\right)\ge V_{-1}\left(x_1,x_2\right)$.
\par 
Now define $\tau_1^*=\inf{\left\{t\ge0:\left(X_t^1,X_t^2\right)\notin \Gamma_2 \right\}}$.  If $\left(X_{\tau_1^*}^1,X_{\tau_1^*}^2\right)\in \Gamma_1$, then $\tau_2^*=$\\$\inf{\left\{t\ge\tau_1^*:\left(X_t^1,X_t^2\right)\in \Gamma_3 \right\}}$.  Otherwise, if $\left(X_{\tau_1^*}^1,X_{\tau_1^*}^2\right)\in \Gamma_3$, then $\tau_2^*=$\\$\inf{\left\{t\ge\tau_1^*:\left(X_t^1,X_t^2\right)\in \Gamma_1 \right\}}$.  Using Dynkin's formula, we obtain 
$$v_0(x_1,x_2)=\mathbb{E}\left[e^{-\rho\tau_1^*}v_0\left(X_{\tau_1^*}^1, X_{\tau_1^*}^2\right)\mathbb{I}_{\{\tau_1^*<\infty\}}\right],$$ 
$$\mathbb{E}\left[e^{-\rho\tau_1^*}v_1\left(X_{\tau_1^*}^1, X_{\tau_1^*}^2\right)\mathbb{I}_{\{\tau_1^*<\infty\}}\right]=\mathbb{E}\left[e^{-\rho\tau_2^*}v_1\left(X_{\tau_2^*}^1, X_{\tau_2^*}^2\right)\mathbb{I}_{\{\tau_2^*<\infty\}}\right],$$
 and $$\mathbb{E}\left[e^{-\rho\tau_1^*}v_{-1}\left(X_{\tau_1^*}^1, X_{\tau_1^*}^2\right)\mathbb{I}_{\{\tau_1^*<\infty\}}\right]=\mathbb{E}\left[e^{-\rho\tau_2^*}v_{-1}\left(X_{\tau_2^*}^1, X_{\tau_2^*}^2\right)\mathbb{I}_{\{\tau_2^*<\infty\}}\right].$$
Thus,
\begin{align*}
v_0(x_1,x_2)=&~\mathbb{E}\left[e^{-\rho\tau_1^*}v_0\left(X_{\tau_1^*}^1, X_{\tau_1^*}^2\right)\mathbb{I}_{\{\tau_1^*<\infty\}}\right]\\
=& ~ \mathbb{E}\left[e^{-\rho\tau_1^*}\left(v_1\left(X_{\tau_1^*}^1, X_{\tau_1^*}^2\right)-\beta_bX_{\tau_1^*}^1+\beta_sX_{\tau_1^*}^2\right)\mathbb{I}_{\{\tau_1^*<\infty\}}\mathbb{I}_{\{u^*=1\}}\right]\\
& ~+ \mathbb{E}\left[e^{-\rho\tau_1^*}\left(v_{-1}\left(X_{\tau_1^*}^1, X_{\tau_1^*}^2\right)+\beta_sX_{\tau_1^*}^1-\beta_bX_{\tau_1^*}^2\right)\mathbb{I}_{\{\tau_1^*<\infty\}}\mathbb{I}_{\{u^*=-1\}}\right]\\
=&~\mathbb{E}\left[e^{-\rho\tau_2^*}\left(\beta_sX_{\tau_2^*}^1-\beta_bX_{\tau_2^*}^2\right)\mathbb{I}_{\{\tau_2^*<\infty\}}\mathbb{I}_{\{u^*=1\}}\right]\\
&~- \mathbb{E}\left[e^{-\rho\tau_2^*}\left(\beta_bX_{\tau_2^*}^1 -\beta_sX_{\tau_2^*}^2\right)\mathbb{I}_{\{\tau_2^*<\infty\}}\mathbb{I}_{\{u^*=-1\}}\right]\\
&~+ \mathbb{E}\left[e^{-\rho\tau_1^*}\left(\beta_sX_{\tau_1^*}^1-\beta_bX_{\tau_1^*}^2\right)\mathbb{I}_{\{\tau_1^*<\infty\}}\mathbb{I}_{\{u^*=-1\}}\right]\\
&~-\mathbb{E}\left[e^{-\rho\tau_1^*}\left(\beta_bX_{\tau_1^*}^1-\beta_sX_{\tau_1^*}^2\right)\mathbb{I}_{\{\tau_1^*<\infty\}}\mathbb{I}_{\{u^*=1\}}\right]\\
=&~J_0\left(x_1,x_2,\tau_1^*,\tau_2^*,u^*\right).
\end{align*}
Similarly,
\begin{align*}
v_{1}(x_1,x_2)=&\mathbb{E}\left[e^{-\rho\tau_0^*}v_{1}\left(X_{\tau_0^*}^1, X_{\tau_0^*}^2\right)\mathbb{I}_{\{\tau_0^*<\infty\}}\right]\\
=&~\mathbb{E}\left[e^{-\rho\tau_0^*}\left(\beta_s X_{\tau_0^*}^2-\beta_bX_{\tau_0^*}^1\right)\mathbb{I}_{\{\tau_0^*<\infty\}}\right]\\
=&~\mathbb{E}\left[-e^{-\rho\tau_0^*}\left(\beta_b X_{\tau_0^*}^1-\beta_sX_{\tau_0^*}^2\right)\mathbb{I}_{\{\tau_0^*<\infty\}}\right]\\
=&~J_1\left(x_1,x_2,\tau_0^*\right),
\end{align*}
and
\begin{align*}
v_{-1}(x_1,x_2)=&\mathbb{E}\left[e^{-\rho\tau_0^*}v_{-1}\left(X_{\tau_0^*}^1, X_{\tau_0^*}^2\right)\mathbb{I}_{\{\tau_0^*<\infty\}}\right]\\
=&~\mathbb{E}\left[e^{-\rho\tau_0^*}\left(\beta_s X_{\tau_0^*}^1-\beta_bX_{\tau_0^*}^2\right)\mathbb{I}_{\{\tau_0^*<\infty\}}\right]\\
=&~J_{-1}\left(x_1,x_2,\tau_0^*\right).
\end{align*}
\end{proof}

\section{A Numerical Example}
\par As in Section 6 above, we consider adjusted closing price data for Walmart (WMT) and Target (TGT) from 2010 to 2020.  The first half of the data is used to calibrate the model, and the second half is used to test the results.  Using a least-squares method, we obtain the following parameters: $\mu_1=0.09696$, $\mu_2=0.14347$, $\sigma_{11}=0.19082$, $\sigma_{12}=0.04036$, $\sigma_{21}=0.04036$, and $\sigma_{22}=0.13988$.  We specify $K=0.001$ and $\rho=0.5$.  Then we find $k^*_1=0.85527$, and $k^*_2=1.32175$.
\par
Next we examine the dependence of $k^*_1$ and $k^*_2$ on the parameters by varying each.  We see that $k^*_1$ and $k^*_2$ both decrease in $\mu_1$.  This leads to a larger buying region, $\Gamma_3$.  On the other hand, both $k^*_1$ and $k^*_2$ increase in $\mu_2$.  This creates a larger $\Gamma_1$ and, hence, encourages early exit.  When varying $\sigma_{11}$ and $\sigma_{22}$, we find that $k^*_1$ decreases while $k^*_2$ increases, in both $\sigma_{11}$ and $\sigma_{22}$.  This leads to a smaller buying zone, $\Gamma_1$, due to the increased risk, as well as a smaller selling zone, $\Gamma_3$, because there is more price movement overall.  However, as $\sigma_{12}=\sigma_{21}$ increases, we find that $k^*_1$ increases, while $k^*_2$ decreases.  The greater correlation leads to a larger $\Gamma_1$, and hence more opportunity for buying, as well as a larger $\Gamma_3$, and hence more opportunity for selling.  Since $r$ represents the rate at which money loses value over time, $k^*_1$ increases in $r$, while $k^*_2$ decreases in $r$, reflecting the fact that we are less likely to want to hold in this case.  Finally, larger transaction costs discourage trading.  Naturally, as $K$ increases, $k^*_1$ decreases and $k^*_2$ increases.  

\begin{table}[H]
\begin{center}
\begin{tabular}{ |c|c|c|c|c|c| } 
\hline
$\mu_1$ & $-0.00304$ & $0.04696$ & $0.09696$ & $0.14696$ & $0.19696$ \\
 \hline
$k^*_1$ & $0.91380$ & $0.89057$ & $0.85527$ & $0.80194$ & $0.72644$  \\ 
\hline
$k^*_2$ & $1.54402$ & $1.42682$ & $1.32175$ & $1.23477$ & $1.17006$  \\
\hline
\end{tabular}
\caption{$k^*_1$ and $k^*_2$ with varying $\mu_1$}
\end{center}
\end{table}

\begin{table}[H]
\begin{center}
\begin{tabular}{ |c|c|c|c|c|c| } 
\hline
$\mu_2$ & $0.04347$ & $0.09347$ & $0.14347$ & $0.19347$ & $0.24347$ \\
\hline
$k^*_1$ & $0.76457$ & $0.81341$ & $0.85527$ & $0.88736$ & $0.91037$  \\ 
\hline
$k^*_2$ & $1.15468$ & $1.21883$ & $1.32175$ & $1.48176$ & $1.72581$  \\ 
\hline
\end{tabular}
\caption{$k^*_1$ and $k^*_2$ with varying $\mu_2$}
\end{center}
\end{table}

\begin{table}[H]
\begin{center}
\begin{tabular}{ |c|c|c|c|c|c| } 
\hline
$\sigma_{11}$ & $0.09082$ & $0.14082$ & $0.19082$ & $0.24082$ & $0.29082$\\
\hline
$k^*_1$ & $0.92069$ & $0.89220$ & $0.85527$ & $0.81532$ & $0.77497$\\
\hline
$k^*_2$ & $1.22784$ & $1.26704$ & $1.32175$ & $1.38652$ & $1.45871$ \\
\hline
\end{tabular}
\caption{$k^*_1$ and $k^*_2$ with varying $\sigma_{11}$}
\end{center}
\end{table}

\begin{table}[H]
\begin{center}
\begin{tabular}{ |c|c|c|c|c|c| } 
\hline
$\sigma_{22}$ & $0.03988$ & $	0.08988$ & $0.13988$ & $0.18988$ & $0.23988$\\
\hline
$k^*_1$ & $0.88356$ & $0.87601$ & $0.85527$ & $0.82593$ & $0.79206$\\
\hline
$k^*_2$ & $1.27943$ & $1.29045$ & $1.32175$ & $1.36871$ & $1.42724$ \\
\hline
\end{tabular}
\caption{$k^*_1$ and $k^*_2$ with varying $\sigma_{22}$}
\end{center}
\end{table}

\begin{table}[H]
\begin{center}
\begin{tabular}{ |c|c|c|c|c|c| } 
\hline
$\sigma_{12}$ & $-0.05964$ & $-0.00964$ & $0.04036$ & $0.09036$ & $0.14036$\\
\hline
$k^*_1$& $0.73242$ & $0.79189$ & $0.85527$ & $0.92029$ & $0.97527$\\
\hline
$k^*_2$ & $1.54345$ & $1.42754$ & $1.32175$ & $1.22837$ & $1.15911$\\
\hline
\end{tabular}
\caption{$k^*_1$ and $k^*_2$ with varying $\sigma_{12}=\sigma_{21}$}
\end{center}
\end{table}

\begin{table}[H]
\begin{center}
\begin{tabular}{ |c|c|c|c|c|c| } 
\hline
$r$ & $0.4$ & $0.45$ & $0.5$ & $0.55$ & $0.6$\\
\hline
$k^*_1$ & $0.84068$ & $0.84858$ & $0.85527$ & $0.86105$ & $0.86611$\\
\hline
$k^*_2$ & $1.40518$ & $1.35725$ & $1.32175$ & $1.29425$ & $1.27222$ \\
\hline
\end{tabular}
\caption{$k^*_1$ and $k^*_2$ with varying $r$}
\end{center}
\end{table}

\begin{table}[H]
\begin{center}
\begin{tabular}{ |c|c|c|c|c|c| } 
\hline
$K$ & $0.0000$ & $0.0005$ & $0.0010$ & $0.0015$ & $0.0020$\\
\hline
$k^*_1$ & $0.85698$ & $0.85613$ & $0.85527$ & $0.85442$ & $0.85356$\\
\hline
$k^*_2$ & $1.31911$ & $1.32043$ & $1.32175$ & $1.32307$ & $1.32439$ \\
\hline
\end{tabular}
\caption{$k^*_1$ and $k^*_2$ with varying $K$}
\end{center}
\end{table}

\clearpage

\end{document}